\documentclass{article}
\usepackage{amsmath}
\usepackage{amsthm}

\usepackage{a4wide}
\usepackage[dvipsnames]{xcolor}
\definecolor{mycolor}{rgb}{0.122, 0.435, 0.698}

\usepackage{booktabs}

\usepackage{hyperref}

\usepackage{enumerate}
\usepackage{units}
\newcommand{\Dx}{{\Delta x}}
\newcommand{\Dt}{{\Delta t}}
\newcommand{\hf}{{\frac{1}{2}}}

\newcommand{\jphf}{{j+\hf}}
\newcommand{\jmhf}{{j-\hf}}

\newcommand*{\Z}{\mathbb{Z}}

\usepackage{tikz}
\usepackage{pgfplots}
\pgfplotsset{width=.53\textwidth,compat=1.12}
\usepgfplotslibrary{fillbetween}

\usepackage{subfig}

\definecolor{skyblue1}{rgb}{0.447,0.624,0.812}
\definecolor{plum1}{rgb}{0.678,0.498,0.659}
\definecolor{scarletred1}{rgb}{0.937,0.161,0.161}

\definecolor{myblue}{HTML}{1e77b4}
\definecolor{myorange}{HTML}{ff7f0f}

\usepackage{booktabs}

\usepackage{amsfonts}

\usepackage{bbm}
\usepackage{mathtools}

\usepackage[obeyFinal,
                color       = mycolor!40!white,
                bordercolor = black,
                textwidth   = 25mm,
                figwidth    = 0.99\linewidth]{todonotes}

\newcommand*{\calF}{\mathcal{F}}
\newcommand*{\bbD}{\mathbb{D}}
\newcommand*{\bbE}{\mathbb{E}}
\newcommand*{\bbP}{\mathbb{P}}

\newcommand*{\diff}{\mathop{}\!\mathrm{d}}
\newcommand*{\TV}{\mathrm{TV}}
\newcommand*{\BV}{\mathrm{BV}}
\newcommand*{\mcom}{\,\text{,}}
\newcommand*{\mper}{\,\text{.}}
\newcommand*{\R}{\mathbb{R}}
\newcommand*{\naturals}{\mathbb{N}}
\newcommand*{\norm}[1]{\left\lVert #1 \right\rVert}
\newcommand*{\abs}[1]{\left\lvert #1 \right\rvert}
\DeclareMathOperator{\sign}{sgn}
\newcommand*{\from}{\colon}
\newcommand*{\compose}{\circ}

\newcommand{\eqdef}{\coloneqq}
\newcommand{\cell}{\mathcal{C}}

\usepackage{cleveref}

\newtheorem{theorem}{Theorem}
\newtheorem{proposition}[theorem]{Proposition}
\newtheorem{remark}[theorem]{Remark}
\newtheorem{definition}[theorem]{Definition}
\newtheorem{lemma}[theorem]{Lemma}
\newtheorem{corollary}[theorem]{Corollary}
\newtheorem{assumption}[theorem]{Assumption}

\numberwithin{equation}{section}
\numberwithin{theorem}{section}

\title{Multilevel Monte Carlo finite volume methods for\\ random conservation laws
with discontinuous flux}

\author{\textsc{Jayesh Badwaik}\thanks{Department of Mathematics, University of W\"urzburg, Germany (\texttt{badwaik.jayesh@gmail.com}, \texttt{klingen@mathematik.uni-wuerzburg.de}).}
\and \textsc{Christian Klingenberg}\footnotemark[1]
\and \textsc{Nils Henrik Risebro}\thanks{Department of Mathematics, University of Oslo, Norway 
  (\texttt{nilshr@math.uio.no}).}
\and \textsc{Adrian\,M. Ruf}\thanks{Seminar for Applied Mathematics, ETH Z\"urich, Switzerland (\texttt{adrian.ruf@sam.math.ethz.ch})
\newline
The work of JB was supported by German Priority Programme 1648
(SPPEXA) and the ModCompShock EU Project. The work of NHR was
performed while visiting the University of W\"{u}rzburg supported by the Giovanni-Prodi Chair Position. JB, CK, and NHR were supported by DAAD (German
Academic Exchange Service) and the Research Council of Norway.
}
}
\date{\today}

\begin{document}

\maketitle
\begin{abstract}
	We consider conservation laws with discontinuous flux where the initial datum, the flux function, and the discontinuous spatial dependency coefficient are subject to randomness.
	We establish a notion of random adapted entropy solutions to these equations and prove well-posedness provided that the spatial dependency coefficient is piecewise constant with finitely many discontinuities. In particular, the setting under consideration allows the flux to change across finitely many points in space whose positions are uncertain.
	We propose a single- and multilevel Monte Carlo method based on a finite volume approximation for each sample.
	Our analysis includes convergence rate estimates of the resulting Monte Carlo and multilevel Monte Carlo finite volume methods as well as error versus work rates showing that the multilevel variant outperforms the single-level method in terms of efficiency. 
	We present numerical experiments motivated by two-phase reservoir simulations for reservoirs with varying geological properties.
\end{abstract}


\section{Introduction}
This paper concerns uncertainty quantification for conservation laws with \emph{discontinuous flux} of the form
\begin{gather}
  \begin{aligned}
    u_t + f(k(x),u)_x =0,& &&x\in\R,\ t>0,\\
    u(x,0)= u_0(x),& &&x\in\R\mper
  \end{aligned}
  \label{eqn: Deterministic Cauchy problem}
\end{gather}
Here, $u\from\R\times [0,\infty)\to\R$ is the unknown and $f\in\mathcal{C}^2(\R^2;\R)$ is the flux function having a possibly \emph{discontinuous} spatial dependency through the coefficient $k$. In particular, we will assume that the initial datum $u_0$ is in $(\mathrm{L}^\infty\cap\BV)(\R)$, the flux $f$ is strictly increasing in $u$, and the coefficient $k$ is piecewise constant with finitely many discontinuities. Going back to~\eqref{eqn: Deterministic Cauchy problem}, this amounts to switching from one $u$-dependent flux to another across finitely many points in space.

Equations of type~\eqref{eqn: Deterministic Cauchy problem} arise in a number of areas of application including vehicle traffic flow in the presence of abruptly varying road conditions (see \cite{lighthill1955kinematic}), polymer flooding in oil recovery (see \cite{Shen:2017aa}), two-phase flow through heterogeneous porous media (see \cite{gimse1992solution,gimse1993note,risebro1991front}), and sedimentation processes (see \cite{diehl1996conservation,burger2003front}).

Even in the absence of flux discontinuities, and even if the initial datum is smooth, solutions of~\eqref{eqn: Deterministic Cauchy problem} develop discontinuities in finite time and for this reason weak solutions are sought. Weak solutions to~\eqref{eqn: Deterministic Cauchy problem} are not unique, so the weak formulation of the problem is augmented with an additional entropy condition. In the case where $x\mapsto f(k(x),u)$ is smooth, uniqueness follows from the classical Kru\v{z}kov entropy conditions \cite{kruvzkov1970first}. In the presence of spatial flux discontinuities, standard Kru\v{z}kov entropy conditions no longer make sense. This difficulty is usually resolved by requiring that Kru\v{z}kov entropy conditions hold away from the spatial flux discontinuities and imposing additional jump conditions along the spatial interfaces \cite{gimse1991riemann,gimse1992solution,diehl1996conservation,klingenberg1995convex,klausen1999stability,Towers1,Towers2,karlsen2003l1,karlsen2004convergence,adimurthi2005optimal,andreianov2011theory} or by adapting the Kru\v{z}kov entropy conditions in a suitable way \cite{Baiti1997,audusse2005uniqueness,Piccoli/Tournus,BadwaikRuf2020,TOWERS20205754,Ruf2020}. In the present paper we will focus on the second approach of so-called adapted entropy solutions for which we need to require that the flux function $f$ is strictly monotone in $u$.

In the last two decades, there has been a large interest in the numerical approximation of entropy solutions of~\eqref{eqn: Deterministic Cauchy problem} under various assumptions on $k$ and $f$.
We refer to \cite{Towers1,Towers2,karlsen2002upwind,karlsen2004convergence,adimurthi2005optimal,Sid2005,ADIMURTHI2007310,wen2008convergence,burger2009engquist,karlsen2017convergence} and \cite{gimse1991riemann,gimse1992solution,gimse1993conservation,klingenberg1995convex,klingenberg2001stability,burger2003front,coclite2005conservation,holden2015front} for a partial list of references regarding finite volume methods respectively the front tracking method.
Specifically, in the adapted entropy framework we want to highlight the results of
\cite{BadwaikRuf2020,TOWERS20205754,ghoshal2020convergence} and \cite{Baiti1997,Piccoli/Tournus,Ruf2020} regarding finite volume methods and the front tracking method.

The classical paradigm for designing efficient numerical schemes assumes that \emph{data for~\eqref{eqn: Deterministic Cauchy problem}, i.e., the initial datum $u_0$, the flux $f$, and the spatial dependency coefficient $k$, are known exactly}.

However, in many situations of practical interest, there is an inherent uncertainty in the modeling and measurement of physical parameters. For example, in two-phase flow through a heterogeneous porous medium the position of the interface between two rock types is typically not known exactly. Often these parameters are only known up to certain statistical quantities of interest like the mean, variance, or higher moments. In such cases, a mathematical framework of~\eqref{eqn: Deterministic Cauchy problem} is required which allows for \emph{random data}.

For standard conservation laws without spatial flux dependency, i.e., for
\begin{gather}
  \begin{aligned}
    u_t + f(u)_x =0,& &&x\in\R,\ t>0,\\
    u(x,0)= u_0(x),& &&x\in\R,
  \end{aligned}
  \label{eqn: Standard conservation law}
\end{gather}
such a framework was developed in a series of papers allowing for random initial datum~\cite{Mishra2012}, random (spatially independent) flux~\cite{Mishra2016}, and even random source terms~\cite{mishra2013multi} and random diffusion~\cite{koley2017multilevel}.

The first aim of the current paper is to \emph{extend this mathematical framework to include scalar conservation laws with discontinuous flux with random discontinuous spatial dependency}. 
To that end, we define random entropy solutions and provide an existence and uniqueness result, which generalizes the well-posedness results for~\eqref{eqn: Standard conservation law} to the case of uncertain initial datum, flux, and discontinuous spatial dependency. In particular, our framework allows for uncertain positions of the flux discontinuities.

The second aim of this paper is to \emph{design fast and robust numerical algorithms for computing the mean of random entropy solutions of conservation laws with discontinuous flux}.
Specifically, we propose and analyze a multilevel combination of Monte Carlo (MC) sampling and a "pathwise" finite volume method (FVM) to approximate the mean of random entropy solutions of conservation laws with discontinuous flux. The multilevel Monte Carlo finite volume method (MLMCFVM) for~\eqref{eqn: Deterministic Cauchy problem} is non-intrusive (in the sense that it requires only repeated applications of existing solvers for input data samples) and easy to implement and to parallelize. Our analysis includes the proof of convergence rates at which the MCFVM and the MLMCFVM converge towards the mean of the random entropy solution of~\eqref{eqn: Deterministic Cauchy problem}. Moreover, we determine the number of MC samples needed to minimize the computational work for a given error tolerance.

We want to emphasize that the framework of adapted entropy solutions and more specifically the setting of the present paper is currently the only setting for which we simultaneously have existence \cite{TOWERS20205754}, uniqueness \cite{audusse2005uniqueness}, stability with respect to the modeling parameters \cite{Ruf2020}, and numerical methods with a provable convergence rate \cite{BadwaikRuf2020,Ruf2020} -- the essential components for an uncertainty quantification framework (cf. \cite{Mishra2016}).

The remainder of this paper is organized as follows. In \Cref{sec: prelim} we introduce preliminary results regarding the MC approximation of Banach space-valued random variables. \Cref{sec: Deterministic conservation laws} is devoted to a review of existence and stability results regarding entropy solutions of (deterministic) conservation laws with discontinuous flux of the form~\eqref{eqn: Deterministic Cauchy problem}. In \Cref{sec: ranconlaw} we introduce random entropy solutions of~\eqref{eqn: Deterministic Cauchy problem} where the initial datum $u_0$, the flux $f$, and the discontinuous coefficient $k$ are subject to randomness. In particular, we prove existence and uniqueness of random entropy solutions. In \Cref{sec: mlmcfvm}, we first review a FVM which was introduced in~\cite{BadwaikRuf2020} for the deterministic problem, prove certain stability estimates, and then extend the FVM to MC as well as MLMC versions for~\eqref{eqn: Deterministic Cauchy problem} with random parameters.
In \Cref{sec: numexp} we perform numerical experiments motivated by two-phase reservoir simulations for reservoirs with varying geological properties to validate our error estimates.
Finally, we summarize the findings of this paper in \Cref{sec: conclusion}.

\section{Preliminaries on the Monte Carlo method}\label{sec: prelim}

We first introduce some preliminary concepts which are needed in the
exposition. To that end, we follow \cite{Ledoux2013} and \cite{VanNeerven2008}, see also \cite[Sec. 2]{koley2017multilevel} and~\cite[Sec. 5]{cox2016convergence}.

Given a probability space $(\Omega, \calF, \bbP)$, a Banach space $V$, and a random variable
$X\from \Omega \to V$ we are interested in approximating the mean $\bbE[X]$ of $X$ via Monte Carlo sampling. To this end, let $(\hat{X}^i)_{i=1}^M$, $i=1,\ldots,M$, be $M$ independent, identically distributed samples  of $X$. Then, the Monte Carlo estimator $E_M[X]$ of $\bbE[X]$ is defined as the sample average
\begin{align*}\label{eqn: mc est}
  E_M[X] \eqdef  \frac{1}{M} \sum_{i=1}^M \hat{X}^i \mper
\end{align*}
We are interested in deriving a rate at which
\begin{equation*}
	\norm{\bbE[X] - E_M[X]}_{\mathrm{L}^q(\Omega;V)} = \bbE[ \norm{\bbE[X] - E_M[X]}_V^q]^{\frac{1}{q}}
\end{equation*}
converges as $M\to\infty$ for some $1\leq q<\infty$ and some Banach space $V$ (typically a Lebesgue space). For general Banach spaces $V$ such convergence rate estimates depend on the type of the Banach space.

\begin{definition}[Banach space of type $q$ {\cite[p. 246]{Ledoux2013}}]
 Assume that $\Omega$ permits a sequence of independent Rademacher random variables $Z_i, i \in \naturals$. We say that a Banach space $V$ is a Banach space of type $1\leq q\leq 2$ if there is a constant $\kappa > 0$ such that for all finite sequences $(x_i)_{i=1}^M \subseteq V$
  \begin{align*}
    \left(\bbE\norm{\sum_{i=1}^M Z_i x_i }_V^q\right)^{\frac{1}{q}} \leq \kappa
    \left(
      \sum_{i=1}^M \norm{x_i}^q_V
    \right)^{\frac{1}{q}} \mper
  \end{align*}
  We will refer to $\kappa$ as the type constant of $V$.
\end{definition}
Every Banach space is a Banach space of type $1$ and every Hilbert space a Banach space of type $2$ \cite[Thm. 9.10]{Ledoux2013}. Moreover, $\mathrm{L}^p$ spaces are Banach spaces of type $q=\min(2,p)$ for $1\leq p<\infty$ \cite[p. 247]{Ledoux2013}.
We will need the following results regarding Lebesgue spaces of functions with values in a Banach space of type $q$.
\begin{lemma}[{\cite[p. 247]{Ledoux2013}}]
  Let $1 \leq r \leq \infty$, $(\Omega, \calF, \bbP)$ be a measure space,
  and $V$ be a Banach space of type $q$. Then the space
  $\mathrm{L}^r(\Omega, V)$ is a Banach space of type $\min(r,q)$.
\end{lemma}
\begin{proposition}[{\cite[Prop. 9.11]{Ledoux2013}}]
  \label{thm: expectation sum of iid rv}
  Let $V$ be a Banach space of type $q$ with type constant $\kappa$.
  Then, for every finite sequence $(X_i)_{i=1}^M$ of independent mean zero random
  variables in $\mathrm{L}^q(\Omega, V)$, we have
  \begin{equation*}
    \bbE\left[\, \norm{ \sum_{i=1}^M X_i}_V^q \,\right] \leq (2\kappa)^q \sum_{i=1}^M \bbE \left[\, \norm{X_i}_V^q\,\right] \mper
  \end{equation*}
\end{proposition}
\begin{corollary}[{\cite[Cor. 2.5]{koley2017multilevel}}]\label{cor: convergence rate mc estimator}
Let $V$ be a Banach space of type $q$ with type constant $\kappa$ and let
$X \in \mathrm{L}^q(\Omega;V)$ be a zero mean random variable. Then for every finite
sequence $(X_i)_{i=1}^M$ of independent, identically distributed random variables with zero mean and with $X_i\sim X$,
we have
\begin{align*}
  \bbE
  \left[
    \norm{E_M[X]}_V^q
  \right] = \bbE\left[ \norm{\frac{1}{M}\sum_{i=1}^M X_i}_V^q \right]
  \leq (2\kappa)^q M^{1-q} \bbE\left[\norm{X}_V^q\right] \mper
\end{align*}
\end{corollary}
We can use \Cref{cor: convergence rate mc estimator} to derive a convergence rate of the Monte Carlo estimator in $\mathrm{L}^q(\Omega;\mathrm{L}^p(\R))$ for random variables in $\mathrm{L}^r(\Omega;\mathrm{L}^p(\R))$.
\begin{theorem}\label{thm: rand var con est}
Let $1 \leq r,p \leq\infty$ and $X \in \mathrm{L}^r(\Omega;\mathrm{L}^p(\R))$, then the Monte Carlo estimator $E_M[X]$
converges towards $\bbE[X]$ in $\mathrm{L}^q(\Omega; \mathrm{L}^p(\R))$ for $q \coloneqq \min\{2,p,r\}$ and we have the bound
\begin{align*}
  \norm{\bbE[X] - E_M[X] }_{\mathrm{L}^q(\Omega;\mathrm{L}^p(\R))}
  \leq C M^{\frac{1-q}{q}} \norm{X}_{\mathrm{L}^q(\Omega;\mathrm{L}^p(\R))} \mper
\end{align*}
\end{theorem}
The proof of this theorem is an adaptation of \cite[Thm. 4.1]{koley2017multilevel}.
\begin{proof}
	We have
	\begin{align*}
		\norm{\bbE[X]-E_M[X]}_{\mathrm{L}^q(\Omega;\mathrm{L}^p(\R))}^q &= \bbE \left[ \norm{\bbE[X]-\frac{1}{M}\sum_{i=1}^M \hat{X}^i}_{\mathrm{L}^p(\R)}^q \right]\\
		&= \bbE \left[ \norm{\frac{1}{M}\sum_{i=1}^M \left(\bbE[X]-\hat{X}^i\right)}_{\mathrm{L}^p(\R)}^q \right]\mper
	\end{align*}
	If we define $Y=\bbE[X]-X$ and $Y_i=\bbE[X]-\hat{X}^i$ we see that $Y$ is in $\mathrm{L}^r(\Omega;\mathrm{L}^p(\R))$ with zero mean and $Y_i$ are i.i.d. random variables with zero mean satisfying $Y_i\sim Y$. Therefore, we can apply \Cref{cor: convergence rate mc estimator} since $\mathrm{L}^r(\Omega;\mathrm{L}^p(\R))$ is of type $\min(2,r,p)$ and $\mathrm{L}^p(\R)$ is of type $\min(2,p)$ and thus in particular also of type $\min(2,r,p)$. Hence,
	\begin{equation*}
		\bbE \left[ \norm{\frac{1}{M}\sum_{i=1}^M \left(\bbE[X]-\hat{X}^i\right)}_{\mathrm{L}^p(\R)}^q \right] \leq (2\kappa)^q M^{1-q} \bbE\left[ \norm{\bbE[X]-X}_{\mathrm{L}^p(\R)}^q \right]
	\end{equation*}
	where $\kappa$ is the type constant of $\mathrm{L}^p(\R)$.
	It remains to show $\bbE\left[ \norm{\bbE[X] - X}_{\mathrm{L}^p(\R)}^q \right] \leq C \bbE\left[\norm{X}_{\mathrm{L}^p(\R)}^q\right]$. This follows from standard estimates and Jensen's inequality in the following way:
	\begin{align*}
		\bbE\left[ \norm{\bbE[X] - X}_{\mathrm{L}^p(\R)}^q \right] &\leq C \bbE\left[ \norm{\bbE[X]}_{\mathrm{L}^p(\R)}^q + \norm{X}_{\mathrm{L}^p(\R)}^q \right]\\
		&\leq C\left( \left(\bbE\left[ \norm{X}_{\mathrm{L}^p(\R)} \right]\right)^q + \bbE\left[ \norm{X}_{\mathrm{L}^p(\R)}^q \right] \right)\\
		&\leq C \bbE\left[\norm{X}_{\mathrm{L}^p(\R)}^q\right]\mper
	\end{align*}
\end{proof}
Note that \Cref{thm: rand var con est} does not imply convergence if $q=1$, i.e., if $r$ or $p$ are equal to~$1$.

\section{Deterministic conservation laws with discontinuous flux}\label{sec: Deterministic conservation laws}

In this section, we present the main existence and stability results for deterministic conservation laws with spatially discontinuous flux from~\cite{Baiti1997},~\cite{TOWERS20205754}, and \cite{Ruf2020}.

We consider the Cauchy problem for conservation laws with discontinuous flux of the form
\begin{gather}
\begin{aligned}
  u_t + f(k(x),u)_x = 0 \mcom& &&x\in \R,\ t>0 \\
  u(x,0)= u_0(x) \mcom& && x\in\R\mper
\end{aligned}
\label{eqn: det conlaw}
\end{gather}
Here, we require that $f$, $k$, and $u_0$ satisfy the following:
\begin{assumption}\label{ass: master assumption}
	We assume that the flux $f\in\mathcal{C}^2(\R^2;\R)$ is strictly monotone in $u$ in the sense that $f_u\geq \alpha >0$, and that $f(k^*,0)=0$ for all $k^*\in\R$. Furthermore, we assume that $k$ is piecewise constant with finitely many discontinuities and that the initial datum $u_0$ is in $(\mathrm{L}^\infty\cap\BV)(\R)$.
\end{assumption}
In the deterministic setting, we consider entropy solutions in the following sense (cf.~\cite{Baiti1997,audusse2005uniqueness}).
For $p\in\R$ we define the function $c_p\from\R\to\R$ through the equation
\begin{equation*}
  f(k(x),c_p(x)) = p,\qquad\text{for all }x\in\R.
\end{equation*}
Since $f_u\geq \alpha>0$ this equation has a unique solution for each $x\in\R$. Note that in the case of piecewise constant $k$ the function $c_p$ is piecewise constant as well.
\begin{definition}[Entropy solution]\label{def: entropy solution}
  We say $u\in\mathcal{C}([0,T];\mathrm{L}^1(\R))\cap\mathrm{L}^\infty((0,T)\times\R)$ is an entropy solution of~\eqref{eqn: det conlaw} if
  \begin{multline*}
    \int_0^T\int_\R \left( |u-c_p(x)|\varphi_t + \sign(u-c_p(x))(f(k(x),u) - f(k(x),c_p(x))) \varphi_x \right)\diff x\diff t\\
    - \int_\R|u(x,T)-c_p(x)|\varphi(x,T)\diff x + \int_\R |u_0(x) - c_p(x)|\varphi(x,0)\diff x \geq 0
  \end{multline*}
  for all $p\in\R$ and for all nonnegative $\varphi\in\mathcal{C}^\infty_c(\R\times[0,T])$.
\end{definition}
Note that a Rankine--Hugoniot-type argument shows that across a discontinuity $\xi$ of $k$ the entropy solution $u$ satisfies the Rankine--Hugoniot condition
\begin{equation}
	f(k(\xi-),u(\xi-,t)) = f(k(\xi+),u(\xi+,t))\qquad\text{for almost every }t\in(0,T)
	\label{eqn: Rankine--Hugoniot condition}
\end{equation}
where $k(\xi\mp)$ and $u(\xi\mp,\cdot)$ denote the left and right traces of $k$ respectively $u$ both of which exist due to~\cite[Rem. 2.3]{andreianov2011theory}.
In our subsequent analysis we will rely on the following two results concerning existence and stability of entropy solutions.
\begin{theorem}[Existence and uniqueness of entropy solutions \cite{Baiti1997,TOWERS20205754,BadwaikRuf2020}]\label{thm: Existence and uniqueness of deterministic entropy solutions}
  Let $f,k$, and $u_0$ satisfy \Cref{ass: master assumption}. Then there exists a unique entropy solution $u$ of \eqref{eqn: det conlaw} which satisfies
  \begin{align}
    \norm{u(\cdot,t)}_{\mathrm{L}^\infty(\R)} \leq \frac{C_f}{\alpha}\norm{u_0}_{\mathrm{L}^\infty(\R)}
    \label{eqn: Linfty bound of deterministic entropy solution}\\
    \TV(u(\cdot,t)) \leq C(\TV(k) + \TV(u_0))
    \notag
    \shortintertext{for all $0\leq t\leq T$ and}
    \TV_{[0,T]}(u(x,\cdot)) \leq C\TV(u_0)
    \notag
  \end{align}
  for all $x\in\R$.
  Here $C_f$ denotes the maximal Lipschitz constant of $f$ and $\alpha$ is as in \Cref{ass: master assumption}.
 \end{theorem}
 \begin{proof}
  The existence and uniqueness statement follows from the theory developed by Baiti and Jenssen~\cite{Baiti1997}. The $\mathrm{L}^\infty$ and $\TV$ bounds follow from~\cite[Thm. 1.4]{TOWERS20205754} and~\cite[Lem. 4.6]{BadwaikRuf2020}.
\end{proof}
 \begin{theorem}[Stability of entropy solutions \cite{Ruf2020}]\label{thm: Stability of deterministic entropy solutions}
 Let $f,k$, and $u_0$ satisfy \Cref{ass: master assumption} and $u$ be the corresponding entropy solution of~\eqref{eqn: det conlaw}.
  If $v$ is the entropy solution of~\eqref{eqn: det conlaw} with flux $g$, coefficient $l$, and initial datum $v_0$ satisfying \Cref{ass: master assumption} then for all $0\leq t\leq T$
  \begin{equation}
    \norm{u(\cdot,t) - v(\cdot,t)}_{\mathrm{L}^1(\R)} \leq C\left( \norm{u_0-v_0}_{\mathrm{L}^1(\R)} + \|k-l\|_{\mathrm{L}^\infty(\R)} + \|f_u - g_u\|_{\mathrm{L}^{\infty}(\R^2;\R)} \right).
    \label{eqn: flux-stability estimate}
  \end{equation}
  In particular, entropy solutions of~\eqref{eqn: det conlaw} satisfy
  \begin{equation*}
    \norm{u(\cdot,t)}_{\mathrm{L}^1(\R)} \leq C\norm{u_0}_{\mathrm{L}^1(\R)}
  \end{equation*}
  for all $0\leq t\leq T$.
\end{theorem}
\begin{proof}
	The stability estimate can be found in~\cite[Thm. 4.1]{Ruf2020}. The $\mathrm{L}^1$ bound follows from the stability estimate~\eqref{eqn: flux-stability estimate} by taking $g=f$, $l=k$, and $v_0=0$.
\end{proof}

\begin{remark}
  We want to mention that the stability result from \Cref{thm: Stability of deterministic entropy solutions} is not only integral in proving existence and uniqueness of random entropy solutions, but can also be used to show well-posedness of Bayesian inverse problems for conservation laws with discontinuous flux \cite{MishraBayesian}.
\end{remark}

\section{Random conservation laws with discontinuous flux}\label{sec: ranconlaw}

We now consider conservation laws with discontinuous flux where the flux $f$, the coefficient $k$, and the initial datum $u_0$ in~\eqref{eqn: det conlaw} are uncertain. To that end, we define appropriate random data $(u_0,k,f)$ in the following sense.
\begin{definition}[Random data]
	Given constants $C_\TV,C_f\in\R$, $\alpha\in(0,\infty)$, $N_k\in\Z$ and given a rectangle $R=R_1\times R_2\subset\R^2$ let $\bbD$ be the Banach space
	\begin{equation*}
		\bbD = (\BV\cap\mathrm{L}^\infty)(\R) \times \mathrm{L}^\infty(\R) \times \mathcal{C}^2(R;\R)
	\end{equation*}
	endowed with the norm
	\begin{equation*}
		\norm{(u_0,k,f)}_\bbD = \norm{u_0}_{\mathrm{L}^1(\R)} + \TV(u_0) + \norm{u_0}_{\mathrm{L}^\infty(\R)} + \norm{k}_{\mathrm{L}^\infty(\R)} + \norm{f}_{\mathcal{C}^2(R;\R)}.
	\end{equation*}
	We say that a strongly measurable map $(u_0,k,f)\from(\Omega,\mathcal{F})\to(\bbD,\mathcal{B}(\bbD))$ is called random data for~\eqref{eqn: det conlaw} if for $\bbP$-a.e. $\omega$
	\begin{align*}
		u_0(\omega;x)\in R_1,& &&\text{for a.e. }x\in\R,\\
		\TV(u_0)\leq C_{\TV} < \infty,&\\
		k(\omega;x)\in R_2,& &&\text{for a.e. }x\in\R,\\
		k(\omega;\cdot) \text{ is pcw.\ constant with at most }N_k \text{ discontinuities,}&\\
		f_u(\omega,k,u)\geq \alpha > 0 \text{ and } f(\omega;k,0) = 0,& &&\text{for all }(k,u)\in R,\\
		\norm{f(\omega;\cdot,\cdot)}_{\mathcal{C}^2(R;\R)} \leq C_f <\infty
	\end{align*}
	such that for $\bbP$-a.e. $\omega$ the data $(u_0(\omega),k(\omega),f(\omega))$ satisfy \Cref{ass: master assumption}.
\end{definition}
We are interested in random entropy solutions of the random conservation law
\begin{gather}
  \begin{aligned}
    \frac{\partial u(\omega;x,t)}{\partial t} + \frac{\partial f(\omega;k(\omega;x),u(\omega;x,t))}{\partial x} = 0,& &&\omega\in\Omega,\ x\in\R,\ t>0,\\
    u(\omega;x,0) = u_0(\omega;x),& &&\omega\in\Omega,\ x\in\R
  \end{aligned}
  \label{eqn: ran conlaw}
\end{gather}
\begin{definition}[Random entropy solution]
  Given random data $(u_0,k,f)\from\Omega\to\bbD$, we say that a random variable $u\from \Omega\to \mathcal{C}([0,T];\mathrm{L}^1(\R))$ is a random entropy solution of~\eqref{eqn: ran conlaw} if it satisfies for all $p\in\R$ and $\bbP$-a.e. $\omega\in\Omega$
  \begin{multline}
    \int_0^T\int_\R \left( |u(\omega;x,t)-c_p(\omega;x)|\varphi_t + q(\omega;u(\omega;x,t)) \right)\diff x\diff t\\
    - \int_\R |u(\omega;x,T) - c_p(\omega;x)|\varphi(x,T)\diff x + \int_\R |u_0(\omega;x) - c_p(\omega;x)|\varphi(x,0)\diff x \geq 0
    \label{eqn: stochastic entropy condition}
  \end{multline}
  for all nonnegative $\varphi\in\mathcal{C}^\infty_c(\R\times[0,T])$. Here we have used the notation
  \begin{equation*}
    q(\omega;u(\omega;x,t))=\sign(u-c_p(\omega;x))(f(\omega;k(\omega;x),u)-f(\omega;k(\omega;x),c_p(\omega;x))).
  \end{equation*}
\end{definition}
We have the following existence and uniqueness result for random entropy solutions of conservation laws with discontinuous flux.
\begin{theorem}[Existence and pathwise uniqueness of random entropy solutions]\label{thm: existence and uniqueness of RESs}
  Let $(u_0,k,f)$ be random data. Then there exists a unique random entropy solution $u\from\Omega\to \mathcal{C}([0,T];\mathrm{L}^1(\R))$ to~\eqref{eqn: ran conlaw} which is pathwise unique, i.e., if the random data $(u_0,k,f)$ and $(v_0,l,g)$ are $\bbP$-versions of each other and $u$ and $v$ are corresponding random entropy solutions then $u$ and $v$ are $\bbP$-versions of each other.
\end{theorem}
\begin{proof}
  Let $S\from\bbD\to \mathcal{C}([0,T];\mathrm{L}^1(\R))$ denote the solution operator from \Cref{thm: Existence and uniqueness of deterministic entropy solutions} that maps (deterministic) $(u_0,k,f)\in\bbD$ to the unique (deterministic) entropy solution $\hat{u}=S(u_0,k,f)$. Because of the stability estimate \eqref{eqn: flux-stability estimate} this solution map is Lipschitz continuous. Now, since the random data $(u_0,k,f)\from \Omega\to \bbD$ is strongly measurable the composition $S\compose (u_0,k,f)\from \Omega\to \mathcal{C}([0,T];\mathrm{L}^1(\R))$ is again strongly measurable 
  (see \cite[Cor. 1.13]{VanNeerven2008}).
  Hence $u = S\compose (u_0,k,f)$ is a strongly measurable map satisfying \eqref{eqn: stochastic entropy condition} $\bbP$-almost surely. Therefore, $u$ is a random entropy solution to \eqref{eqn: ran conlaw}.
  
  Regarding uniqueness of random entropy solutions, let $(u_0,k,f)$ and $(v_0,l,g)$ be $\bbP$-versions of each other, i.e., $\|(u_0(\omega),k(\omega),f(\omega))-(v_0(\omega),l(\omega),g(\omega))\|_\bbD =0$ for $\bbP$-a.e. $\omega\in\Omega$, and $u$ and $v$ corresponding random entropy solutions. Then, the Lipschitz continuity of the solution operator $S$ gives
  \begin{equation*}
    \norm{u(\omega)-v(\omega)}_{\mathcal{C}([0,T];\mathrm{L}^1(\R))} \leq \|(u_0(\omega),k(\omega),f(\omega))-(v_0(\omega),l(\omega),g(\omega))\|_\bbD = 0.
  \end{equation*}
  Thus, we have $u(\omega)=v(\omega)$ in $\mathcal{C}([0,T];\mathrm{L}^1(\R))$ for $\bbP$-a.e. $\omega\in\Omega$ which is pathwise uniqueness.
\end{proof}
Note that \Cref{thm: existence and uniqueness of RESs} generalizes the existence result of random entropy solutions of~\cite{Mishra2016} for fluxes which are strictly monotone in $u$ since the present setting allows for a discontinuous spatial dependency of the flux.
\begin{remark}
  All existence and continuous dependence results stated so far apply to the deterministic Cauchy problem~\eqref{eqn: det conlaw}. By the usual arguments, verbatim the same results hold for entropy solutions on bounded intervals $D\subset\R$ as well, provided periodic boundary conditions are enforced.
\end{remark}
The following probabilistic bound will be important in the numerical approximation of random entropy solutions on bounded domains.
\begin{lemma}\label{lem: r-th moment estimate of the entropy solution}
  Let $(u_0,k,f)$ be random data and $D\subset \R$ a bounded interval. Let further $u_0\in \mathrm{L}^r(\Omega;\mathrm{L}^\infty(D))$, for some $1\leq r \leq \infty$. Then the random entropy solution $u$ of~\eqref{eqn: ran conlaw} is in $\mathrm{L}^r(\Omega;\mathcal{C}([0,T];\mathrm{L}^p(D)))$ for all $1\leq p\leq \infty$. In particular,
  \begin{equation*}
    \norm{u(\cdot,t)}_{\mathrm{L}^r(\Omega;\mathrm{L}^p(D))} \leq C \norm{u_0}_{\mathrm{L}^r(\Omega;\mathrm{L}^\infty(D))}
  \end{equation*}
  for all $0\leq t\leq T$.
\end{lemma}
\begin{proof}
  On bounded domains $D$ we have
  \begin{equation*}
    \norm{u(\cdot,t)}_{\mathrm{L}^p(D)} \leq |D|^{\frac{1}{p}} \norm{u(\cdot,t)}_{\mathrm{L}^\infty(D)}
  \end{equation*}
  and thus using the $\mathrm{L}^\infty$-bound~\eqref{eqn: Linfty bound of deterministic entropy solution} we have for all $0\leq t\leq T$
  \begin{align*}
    \norm{u(\cdot,t)}_{\mathrm{L}^r(\Omega;\mathrm{L}^p(D))}^r &= \int_\Omega \norm{u(\cdot,t)}_{\mathrm{L}^p(D)}^r\diff\bbP\\
    &\leq C \int_\Omega \norm{u(\cdot,t)}_{\mathrm{L}^\infty(D)}^r\diff\bbP\\
    &\leq C \int_\Omega \norm{u_0}_{\mathrm{L}^\infty(D)}^r\diff\bbP\\
    &= C\norm{u_0}_{\mathrm{L}^r(\Omega;\mathrm{L}^\infty(D))}^r
  \end{align*}
  which proves the claim.
\end{proof}

\section{Numerical approximation of random entropy solutions}
\label{sec: mlmcfvm}

In this section, we want to approximate the expectation $\bbE[u(\cdot,t)]$ of a random entropy solution~$u$ of the random conservation law with discontinuous flux \eqref{eqn: ran conlaw}. On the one hand, we will use the Monte Carlo and multilevel Monte Carlo method to approximate in the stochastic domain $\Omega$. On the other hand, since in general exact solutions to \eqref{eqn: ran conlaw} are not at hand, we will approximate in the physical domain $\R\times [0,T]$ by a finite volume method. To this end, we use a modified version of monotone finite volume methods for conservation laws introduced in~\cite{BadwaikRuf2020} which appropriately addresses the presence of the discontinuous parameter $k$.

The resulting approximation error introduced by the Monte Carlo method depends on the number of samples
used, while the error introduced by the finite volume method depends
on the resolution of the grid.
In the following subsections, we will review the finite volume method for the deterministic problem, detail how to combine it with the Monte Carlo and multilevel Monte Carlo method and prove error estimates for the resulting Monte Carlo and multilevel Monte Carlo finite volume method.

\subsection{Finite volume methods for conservation laws with discontinuous flux}
We will first consider the (deterministic) conservation law with discontinuous flux \eqref{eqn: det conlaw} and present a class of finite volume methods introduced in~\cite{BadwaikRuf2020}.

We discretize the domain $\R\times[0,T]$ using the spatial and temporal grid discretization parameters $\Dx$ and $\Dt$. The resulting grid cells we denote by $\cell_j=(x_\jmhf,x_\jphf)$ in space and $\cell^n=[t^n,t^{n+1})$ in time for points $x_\jphf$, such that $x_\jphf-x_\jmhf=\Dx$, $j\in\Z$, and $t^n = n\Dt$ for $n=0,\ldots,M+1$.

For a given coefficient $k$ we denote by $\xi_i$, $i=1,\ldots,N$, its discontinuities and by $D_i=(\xi_i,\xi_{i+1})$, $i=0,\ldots,N$, the subdomains where $k$ is constant. Here we have used the notation $\xi_0=-\infty$ and $\xi_{N+1}=+\infty$. Furthermore, we will write
\begin{equation*}
	f^{(i)} = f(k(x),\cdot),\qquad \text{for }x\in D_i,\ i=0,\ldots,N.
\end{equation*}

In the following, we will assume that the grid is aligned in such a way that all discontinuities of $k$ lie on cell interfaces, i.e., $\xi_i=x_{P_i -\hf}$ for some integers $P_i$, $i=1,\ldots,N$. In general, this can be achieved by considering a globally nonuniform grid that is uniform on each $D_i$ and taking $\Dx = \max_{i=0,\ldots,N} \Dx_i$ where $\Dx_i$ is the grid discretization parameter in $D_i$.

We consider two-point numerical fluxes $F(u,v)$ that have the upwind property such that if $f'\geq 0$ (which is the setting of the present paper), we have $F(u,v)=f(v)$. This includes the upwind flux, the Godunov flux, and the Engquist--Osher flux. The finite volume method we consider is the following \cite{BadwaikRuf2020}:
\begin{gather}
  \begin{aligned}
    u_j^0 = \frac{1}{\Dx}\int_{\cell_j} u_0(x)\diff x,& &&j\in\Z,\\
    u_j^{n+1} = u_j^n - \lambda\left( f^{(i)}(u_j^n) - f^{(i)}(u_{j-1}^n) \right),& &&n\geq 0,\ P_i < j < P_{i+1},\ 0\leq i\leq N,\\
    u_{P_i}^{n+1} = \left(f^{(i)}\right)^{-1}\left(f^{(i-1)}\left(u_{P_i -1}^{n+1}\right)\right),& &&n\geq 0,\ 0<i\leq N,
  \end{aligned}
  \label{eqn: FVM}
\end{gather}
where $P_0=-\infty$, $P_{N+1}=+\infty$, and $\lambda = \Dt/\Dx$. We assume that the grid discretization parameters satisfy the following CFL condition:
\begin{equation}
  \max_i \max_u \left(f^{(i)}\right)'(u) \lambda \leq 1.
  \label{eqn: CFL condition}
\end{equation}
Note that the last line of~\eqref{eqn: FVM} represents a discrete version of the Rankine--Hugoniot condition~\eqref{eqn: Rankine--Hugoniot condition}. Here, we use the ghost cells $\cell_{P_i}$, $i=1,\ldots,N$ to explicitly enforce the Rankine--Hugoniot condition on the discrete level.

With the sequence of cell averages $(u_j^n)_{j,n}$ we associate the piecewise constant function $u_\Dx(x,t)$ given by
\begin{equation*}
	u_\Dx(x,t) = u_j^n,\qquad (x,t)\in\cell_j\times\cell^n.
\end{equation*}
The following lemma shows that the finite volume method is stable in $\mathrm{L}^\infty$ and $\mathrm{L}^1$.
\begin{lemma}[Stability of the finite volume method]
  If the numerical scheme~\eqref{eqn: FVM} satisfies the CFL condition~\eqref{eqn: CFL condition} we have the following stability estimates:
  \begin{align}
    \norm{u_\Dx(\cdot,t)}_{\mathrm{L}^\infty(\R)} \leq \frac{C_f}{\alpha}\norm{u_0}_{\mathrm{L}^\infty(\R)}
    \label{eqn: Linfty bound of FVM}
    \shortintertext{and}
  	\norm{u_\Dx(\cdot,t)}_{\mathrm{L}^1(\R)} \leq \norm{u_0}_{\mathrm{L}^1(\R)} + C\TV(u_0)\Dx.
    \notag
  \end{align}
\end{lemma}
\begin{proof}
	\begin{enumerate}[(1)]
		\item We first prove the $\mathrm{L}^\infty$-bound. To that end, we show by induction over $i=0,\ldots,N$ that
		\begin{equation}
			u_j^n \leq \max_{m=0,\ldots,i} \sup_{l=P_m,\ldots,P_{m+1}-1} \left(f^{(i)}\right)^{-1}\left(f^{(m)} (u_l^0)\right)
			\label{eqn: induction hypothesis}
		\end{equation}
		for all $j=P_i,\ldots,P_{i+1}-1$ and $n=0,\ldots,M+1$. For $i=0$, standard techniques for finite volume methods for conservation laws show
		\begin{equation*}
			u_j^n \leq \max \{ u_{j-1}^{n-1}, u_j^{n-1} \} \leq \ldots \leq \sup_{l<P_1} u_l^0.
		\end{equation*}
		Assume now that~\eqref{eqn: induction hypothesis} holds for some $i\in\{0,\ldots,N-1\}$ and all $j=P_i,\ldots,P_{i+1}-1$ and $n=0,\ldots,M+1$. Then we have for $j=P_{i+1}$
		\begin{align*}
			u_{P_{i+1}}^n &= \left(f^{(i+1)}\right)^{-1}\left( f^{(i)}(u_{P_{i+1}-1}^n) \right)\\
			&\leq \max_{m=0,\ldots,i} \sup_{l=P_m,\ldots,P_{m+1}-1} \left( f^{(i+1)} \right)^{-1} \left( f^{(m)} (u_l^0) \right).
		\end{align*}
		On the other hand, for $j\in\{P_{i+1}+1,\ldots,P_{i+2}-1\}$ we have as before
		\begin{align*}
			u_j^n &\leq \max\{ u_{j-1}^{n-1},\ldots,u_{j-1}^1, u_{j-1}^0, u_j^0 \}\\
			&\leq\ldots\leq \max\{ u_{P_{i+1}}^{n-(j-P_{i+1})},\ldots, u_{P_{i+1}}^1, u_{P_{i+1}}^0,\ldots,u_j^0 \}.
		\end{align*}
		By combining both estimates, we obtain for $j\in\{P_{i+1},\ldots,P_{i+2}-1\}$
		\begin{align*}
			u_j^n &\leq \max\left\{ \max_{l=P_{i+1},\ldots,P_{i+2}-1} u_l^0, \max_{m=0,\ldots,i} \sup_{l=P_m,\ldots,P_{m+1}-1} \left(f^{(i+1)}\right)^{-1}\left( f^{(m)}(u_l^0) \right) \right\}\\
			&= \max_{m=0,\ldots,i+1} \sup_{l=P_m,\ldots,P_{m+1}-1} \left(f^{(i+1)}\right)^{-1}\left( f^{(m)}(u_l^0) \right)
		\end{align*}
		which completes the induction. By taking absolute values in~\eqref{eqn: induction hypothesis} we get for $j\in\Z$
		\begin{equation*}
			|u_j^n| \leq \frac{1}{\alpha} \max_{i=0,\ldots,N}\norm{f^{(i)}}_{\mathrm{Lip}} \norm{u_0}_{\mathrm{L}^\infty(\R)}.
		\end{equation*}
		Taking the supremum over $j$ yields the $\mathrm{L}^\infty$-bound~\eqref{eqn: Linfty bound of FVM}.
		\item In order to prove the $\mathrm{L}^1$-bound note that we have the discrete entropy inequalities
		\begin{equation*}
			|u_j^{n+1} - c| - |u_j^n - c| + \lambda \left( q_j^{(i),n} - q_{j-1}^{(i),n} \right) \leq 0,\qquad i=0,\ldots,N,\ j=P_i +1,\ldots,P_{i+1}-1
		\end{equation*}
		for all $c\in\R$ (see~\cite{BadwaikRuf2020}). Here, we have denoted $q_j^{(i),n} = |f^{(i)}(u_j^n) - f^{(i)}(c)|$. Taking $c=0$ and summing over $j\in\Z\setminus\{P_1,\ldots,P_N\}$ yields
		\begin{equation*}
			\sum_{j\neq P_i} |u_j^{n+1}| \leq \sum_{j\neq P_i} |u_j^n| - \lambda\sum_{i=0}^N\sum_{j=P_i +1}^{P_{i+1}-1}\left( q_j^{(i),n} - q_{j-1}^{(i),n}\right) = \sum_{j\neq P_i} |u_j^n|.
		\end{equation*}
		Therefore, we have
		\begin{align*}
			\sum_{j\in\Z} |u_j^{n+1}| &\leq \sum_{j\in\Z}|u_j^n| + \sum_{i=1}^N (|u_{P_i}^{n+1}| - |u_{P_i}^n|)\\
			&\leq \sum_{j\in\Z}|u_j^n| + \sum_{i=1}^N \frac{1}{\alpha}\norm{f^{(i-1)}}_{\mathrm{Lip}}\abs{u_{P_i -1}^{n+1} - u_{P_i -1}^n}
		\end{align*}
		and hence
		\begin{equation*}
			\sum_{j\in\Z} |u_j^{n+1}| \leq \sum_{j\in\Z}|u_j^0| + \sum_{i=0}^N \frac{1}{\alpha}\norm{f^{(i-1)}}_{\mathrm{Lip}} \sum_{m=0}^{n} \abs{u_{P_i -1}^{m+1} - u_{P_i -1}^m}.
		\end{equation*}
		In~\cite[Lem. 4.6]{BadwaikRuf2020} it was shown that for all $i=0,\ldots,N$ we have
		\begin{equation*}
			\sum_{m=0}^n \abs{u_{P_i -1}^{m+1} - u_{P_i -1}^m} \leq C \TV(u_0)
		\end{equation*}
		which together with the foregoing estimate finally yields
		\begin{equation*}
			\norm{u_\Dx(\cdot,t)}_{\mathrm{L}^1(\R)} \leq \norm{u_0}_{\mathrm{L}^1(\R)} + C\TV(u_0)\Dx.
		\end{equation*}
	\end{enumerate}
\end{proof}
In order to prove error estimates of the Monte Carlo and multilevel Monte Carlo finite volume method we will need the following convergence rate estimate.
\begin{theorem}[Convergence rate of the finite volume method \cite{BadwaikRuf2020}]
	Let $f,k$, and $u_0$ satisfy \Cref{ass: master assumption} and the discretization parameters satisfy the CFL condition~\eqref{eqn: CFL condition}. Then the finite volume approximation $u_\Dx$ given by the scheme~\eqref{eqn: FVM} converges towards the unique entropy solution $u$ of~\eqref{eqn: ran conlaw} almost everywhere and in $\mathrm{L}^1(\R\times(0,T))$. In particular, we have the following convergence rate estimate
	\begin{equation}
		\norm{u(\cdot,t) - u_\Dx(\cdot,t)}_{\mathrm{L}^1(\R)} \leq C \Dx^\hf
    \label{eqn: Convergence rate estimate for FVM}
	\end{equation}
	for all $0\leq t\leq T$.
\end{theorem}
Note that the convergence rate estimate~\eqref{eqn: Convergence rate estimate for FVM} is optimal in the sense that the exponent $\hf$ cannot be improved without further assumptions on the initial datum \cite{BadwaikRuf2020} (see~\cite{Ruf2019} for an overview of the literature regarding optimal convergence rates of finite volume methods for conservation laws without spatial dependency).
\begin{remark}
	Reasoning as for entropy solutions, the finite volume approximation satisfies
	\begin{equation*}
		\norm{u_\Dx(\cdot,t)}_{\mathrm{L}^p(D)} \leq |D|^{\frac{1}{p}} \norm{u_\Dx(\cdot,t)}_{\mathrm{L}^\infty(D)} \leq C\norm{u_0}_{\mathrm{L}^\infty(D)}
	\end{equation*}
	for all $1\leq p\leq \infty$. Like in \Cref{lem: r-th moment estimate of the entropy solution}, this translates into the following probabilistic bound:
	\begin{equation}
		\norm{u_\Dx(\cdot,t)}_{\mathrm{L}^r(\Omega;\mathrm{L}^p(D))} \leq C \norm{u_0}_{\mathrm{L}^r(\Omega;\mathrm{L}^\infty(D))}
		\label{eqn: r-th moment bound of FVM}
	\end{equation}
	for all $0\leq t\leq T$ and $1\leq p\leq \infty$.
\end{remark}
For the rest of this paper, we will consider entropy solutions on a bounded interval $D\subset \R$ with periodic boundary conditions. With the usual arguments, all previous results concerning entropy solutions and their finite volume approximations carry over to this setting verbatim. Note that restricting ourselves to a bounded domain will enable us to prove error estimates of the Monte Carlo and multilevel Monte Carlo finite volume method also in $\mathrm{L}^2(\Omega;\mathrm{L}^1(D))$ (cf. \cite{risebro2018correction}).

\subsection{Monte Carlo finite volume method}

We now consider the random conservation law with discontinuous flux~\eqref{eqn: ran conlaw} and introduce and analyze the Monte Carlo finite volume method.

Given $M\in\naturals$, we generate $M$ independent and identically distributed samples $(\hat{f}^i,\hat{k}^i,\hat{u}_0^i)_{i=1}^M$ of given random data $(u_0,k,f)$. Let now $\hat{u}_\Dx^i(\cdot,t)$, $i=1,\ldots,M$, denote the numerical solutions generated by the finite volume method \eqref{eqn: FVM} at time $t$ corresponding to the sample $(\hat{f}^i,\hat{k}^i,\hat{u}_0^i)$. Then, the $M$-sample MCFVM approximation to $\bbE[u(\cdot,t)]$ is defined as
\begin{equation*}
	E_M[u_\Dx(\cdot,t)] = \frac{1}{M}\sum_{i=1}^M \hat{u}_\Dx^i(\cdot,t).
\end{equation*}
As mentioned earlier the approximation error of the MCFVM has a component coming from the statistical sampling error and one from the deterministic discretization error. We will make this statement precise in the following theorem.
\begin{theorem}[MCFVM error estimate]\label{thm: MCFVM error estimate}
	Let $(u_0,k,f)$ be random data and $u$ the corresponding random entropy solution of \eqref{eqn: ran conlaw}. Assume that $u_0$ satisfies the $r$-th moment condition
	\begin{equation*}
		\norm{u_0}_{\mathrm{L}^r(\Omega;\mathrm{L}^\infty(D))} <\infty
	\end{equation*}
	for some $1< r\leq\infty$. Assume further that we are given a FVM~\eqref{eqn: FVM} such that the CFL condition~\eqref{eqn: CFL condition} holds.
	Then, for each $1\leq p\leq \infty$ and $0\leq t\leq T$ and for $q=\min(2,r)>1$, the MCFVM approximation satisfies the error estimate
	\begin{equation}
		\norm{\bbE[u(\cdot,t)] - E_M[u_\Dx(\cdot,t)]}_{\mathrm{L}^q(\Omega;\mathrm{L}^p(D))} \leq C \left( M^{\frac{1-q}{q}} \norm{u_0}_{\mathrm{L}^r(\Omega;\mathrm{L}^\infty(D))}  + \norm{u_0}_{\mathrm{L}^r(\Omega;\mathrm{L}^\infty(D))}^{1-\frac{1}{p}} \Dx^{\frac{1}{2p}}\right)\mper
		\label{eqn: MCFVM error estimate}
	\end{equation}
	In particular, the MCFVM approximation converges towards $\bbE[u(\cdot,t)]$ in $\mathrm{L}^q(\Omega;\mathrm{L}^p(D))$ as $M\to\infty$ and $\Dx\to 0$.
\end{theorem}
\begin{proof}
	We use the triangle inequality to get
	\begin{multline}
		\norm{\bbE[u(\cdot,t)] - E_M[u_\Dx(\cdot,t)]}_{\mathrm{L}^q(\Omega;\mathrm{L}^p(D))}\\\leq \norm{\bbE[u(\cdot,t)] - E_M[u(\cdot,t)]}_{\mathrm{L}^q(\Omega;\mathrm{L}^p(D))} + \norm{E_M[u(\cdot,t)] - E_M[u_\Dx(\cdot,t)]}_{\mathrm{L}^q(\Omega;\mathrm{L}^p(D))}
		\label{eqn: Triangle inequality in MCFVM proof}
	\end{multline}
	and estimate the resulting two terms separately. For the first term in~\eqref{eqn: Triangle inequality in MCFVM proof}, we distinguish the two cases $p\geq q$ and $p<q$.
\begin{enumerate}[(1)]
		\item We first consider the case $p\geq q$. According to \Cref{lem: r-th moment estimate of the entropy solution} we have
	\begin{equation*}
		\norm{u(\cdot,t)}_{\mathrm{L}^r(\Omega;\mathrm{L}^p(D))} \leq C \norm{u_0}_{\mathrm{L}^r(\Omega;\mathrm{L}^\infty(D))}
	\end{equation*}
	and thus $u(\cdot,t)\in \mathrm{L}^r(\Omega;\mathrm{L}^p(D))$. Therefore, we can apply \Cref{thm: rand var con est} to get
	\begin{align*}
		\norm{\bbE[u(\cdot,t)] - E_M[u(\cdot,t)]}_{\mathrm{L}^q(\Omega;\mathrm{L}^p(D))} &\leq C M^{\frac{1-q}{q}} \norm{u(\cdot,t)}_{\mathrm{L}^q(\Omega;\mathrm{L}^p(D))}\\
		&\leq C M^{\frac{1-q}{q}} \norm{u(\cdot,t)}_{\mathrm{L}^r(\Omega;\mathrm{L}^p(D))}\\
		&\leq C M^{\frac{1-q}{q}} \norm{u_0}_{\mathrm{L}^r(\Omega;\mathrm{L}^\infty(D))}\mper
	\end{align*}
	\item In the case $p<q$, we can apply H\"older's inequality to estimate
	\begin{equation*}
		\norm{\bbE[u(\cdot,t)] - E_M[u(\cdot,t)]}_{\mathrm{L}^q(\Omega;\mathrm{L}^p(D))} \leq C \norm{\bbE[u(\cdot,t)] - E_M[u(\cdot,t)]}_{\mathrm{L}^q(\Omega;\mathrm{L}^q(D))}\mper
	\end{equation*}
	Again, we want to employ \Cref{thm: rand var con est}.
  To that end, we note that because of \Cref{lem: r-th moment estimate of the entropy solution} and the fact that $q\leq r$ we have
  \begin{align*}
    \norm{u(\cdot,t)}_{\mathrm{L}^q(\Omega;\mathrm{L}^q(D))} &\leq C\norm{u_0}_{\mathrm{L}^q(\Omega;\mathrm{L}^\infty(D))}\\
    &\leq C \norm{u_0}_{\mathrm{L}^r(\Omega;\mathrm{L}^\infty(D))}
  \end{align*}
  and therefore $u(\cdot,t)\in\mathrm{L}^q(\Omega;\mathrm{L}^q(D))$ and we can apply \Cref{thm: rand var con est} to get
	\begin{align*}
		\norm{\bbE[u(\cdot,t)] - E_M[u(\cdot,t)]}_{\mathrm{L}^q(\Omega;\mathrm{L}^q(D))} &\leq C M^{\frac{1-q}{q}} \norm{u(\cdot,t)}_{\mathrm{L}^q(\Omega;\mathrm{L}^q(D))}\\
		&\leq C M^{\frac{1-q}{q}} \norm{u_0}_{\mathrm{L}^r(\Omega;\mathrm{L}^\infty(D))}\mper
	\end{align*}
\end{enumerate}
Hence, for all $1\leq p\leq \infty$, we get
\begin{equation*}
	\norm{\bbE[u(\cdot,t)] - E_M[u(\cdot,t)]}_{\mathrm{L}^q(\Omega;\mathrm{L}^p(D))} \leq C M^{\frac{1-q}{q}} \norm{u_0}_{\mathrm{L}^r(\Omega;\mathrm{L}^\infty(D))}\mper
\end{equation*}
On the other hand, for the second term in \eqref{eqn: Triangle inequality in MCFVM proof} we can use the triangle inequality and the linearity of the expected value to obtain
\begin{align*}
	\norm{E_M[u(\cdot,t)] - E_M[u_\Dx(\cdot,t)]}_{\mathrm{L}^q(\Omega;\mathrm{L}^p(D))} &\leq \frac{1}{M}\sum_{i=1}^M \norm{\hat{u}^i(\cdot,t) - \hat{u}_\Dx^i(\cdot,t)}_{\mathrm{L}^q(\Omega;\mathrm{L}^p(D))}\\
	&= \norm{u(\cdot,t) -u_\Dx(\cdot,t)}_{\mathrm{L}^q(\Omega;\mathrm{L}^p(D))}\mper
\end{align*}
Using the interpolation inequality between $\mathrm{L}^1$ and $\mathrm{L}^\infty$, the $\mathrm{L}^\infty$-bound for both $u(\cdot,t)$ and $u_\Dx(\cdot,t)$ (see \eqref{eqn: Linfty bound of deterministic entropy solution} respectively \eqref{eqn: Linfty bound of FVM}), and the convergence rate estimate \eqref{eqn: Convergence rate estimate for FVM}, we get
  \begin{align*}
    \norm{u(\cdot,t) -u_\Dx(\cdot,t)}_{\mathrm{L}^q(\Omega;\mathrm{L}^p(D))} &\leq \norm{u(\cdot,t) - u_\Dx(\cdot,t)}_{\mathrm{L}^q(\Omega;\mathrm{L}^1(D))}^{\frac{1}{p}} \norm{u(\cdot,t) - u_\Dx(\cdot,t)}_{\mathrm{L}^q(\Omega;\mathrm{L}^\infty(D))}^{1-\frac{1}{p}}\\
    &\leq C\norm{u_0}_{\mathrm{L}^r(\Omega;\mathrm{L}^\infty(D))}^{1-\frac{1}{p}}\Dx^{\frac{1}{2p}}\mper
  \end{align*}
	which completes the proof.
\end{proof}

\subsection{Multilevel Monte Carlo finite volume method}
Instead of just considering Monte Carlo samples of a single fixed resolution of the finite volume method, we now detail the corresponding multilevel variant -- the multilevel Monte Carlo finite volume method. The key ingredient is simultaneous MC sampling on different levels of resolution of the finite volume method with level-dependent numbers $M_l$ of MC samples.

To that end, we generate a sequence of finite volume approximations $U(\cdot,t)\eqdef(u_l(\cdot,t))_{l=0}^L$ on grids with cell sizes $\Dx_l$ and time steps $\Dt_l$ (subject to the CFL condition~\eqref{eqn: CFL condition}) and set $u_{\Dx_{-1}}(\cdot,t)=0$. Then, we have
\begin{equation*}
	\bbE[u_{\Dx_L}(\cdot,t)] = \bbE\left[ \sum_{l=0}^L (u_{\Dx_l}(\cdot,t) - u_{\Dx_{l-1}}(\cdot,t)) \right] = \sum_{l=0}^L \bbE[u_{\Dx_l}(\cdot,t) - u_{\Dx_{l-1}}(\cdot,t)]\mper
\end{equation*}
We now approximate each term $\bbE[u_{\Dx_l}(\cdot,t) - u_{\Dx_{l-1}}(\cdot,t)]$ by a Monte Carlo estimator with $M_l$ samples. The resulting MLMCFVM approximation to $\bbE[u(\cdot,t)]$ then is
\begin{equation}
	E^L[U(\cdot,t)] = \sum_{l=0}^L E_{M_l}\left[ u_{\Dx_l}(\cdot,t) - u_{\Dx_{l-1}}(\cdot,t) \right]\mper
	\label{eqn: MLMCFVM definition}
\end{equation}
In the following convergence analysis, we will assume for simplicity that $\Dx_l = 2^{-l}\Dx_0$, $l=0,\ldots,L$, for some $\Dx_0 >0$.

As for the MCFVM, we want to obtain a rate at which $E^L[U(\cdot,t)]$ converges towards $\bbE[u(\cdot,t)]$ in terms of the number of MC samples $M_l$ and the spatial resolution $\Dx_l$ on each level $l=0,\ldots,L$.
\begin{theorem}[MLMCFVM error estimate]
  Let $L>0$, $(u_0,k,f)$ be random data, and $u$ the corresponding random entropy solution of \eqref{eqn: ran conlaw}. Assume that $u_0$ satisfies
  \begin{equation*}
    \norm{u_0}_{\mathrm{L}^r(\Omega;\mathrm{L}^\infty(D))} <\infty
  \end{equation*}
  for some $1 < r\leq\infty$. Assume further that we are given a FVM~\eqref{eqn: FVM} such that the CFL condition~\eqref{eqn: CFL condition} holds.
  Then, for each $0\leq t\leq T$, for any sequence $(M_l)_{l=0}^L$ of sample sizes at mesh level $l$ the MLMCFVM approximation~\eqref{eqn: MLMCFVM definition} satisfies the following error estimate for $q=\min(2,r)>1$
  \begin{multline}
    \norm{\bbE[u(\cdot,t)] - E^L[U(\cdot,t)]}_{\mathrm{L}^q(\Omega;\mathrm{L}^p(\R))}\\
    \leq C \left( \norm{u_0}_{\mathrm{L}^1(\Omega;\mathrm{L}^\infty(D))}^{1-\frac{1}{\widetilde{p}}} \Dx_L^{\frac{1}{2p}} + \norm{u_0}_{\mathrm{L}^q(\Omega;\mathrm{L}^\infty(D))} M_0^{\frac{1-q}{q}} +\norm{u_0}_{\mathrm{L}^q(\Omega;\mathrm{L}^\infty(D))}^{1-\frac{1}{\widetilde{p}}}\sum_{l=0}^L M_l^{\frac{1-q}{q}} \Dx_l^{\frac{1}{2\widetilde{p}}} \right)
    \label{eqn: MLMCFVM error estimate}
  \end{multline}
  where $\widetilde{p}=\max(p,q)$.
  In particular, for fixed $L$ the MLMCFVM approximation $E^L[U(\cdot,t)]$ converges towards $\bbE[u(\cdot,t)]$ in $\mathrm{L}^q(\Omega;\mathrm{L}^p(D))$ as $M_l\to\infty$ and $\Dx_0\to 0$.
\end{theorem}
\begin{proof}
  Using the triangle inequality and the linearity of the expectation, we get
  \begin{align*}
    \| \bbE[u(\cdot,t)] &- E^L[U(\cdot,t)] \|_{\mathrm{L}^q(\Omega;\mathrm{L}^p(D))}\\
    &\leq \norm{\bbE[u(\cdot,t)] - \bbE[u_{\Dx_L}(\cdot,t)]}_{\mathrm{L}^q(\Omega;\mathrm{L}^p(D))} + \norm{\bbE[u_{\Dx_L}(\cdot,t)] - E^L[U(\cdot,t)]}_{\mathrm{L}^q(\Omega;\mathrm{L}^p(D))}\\
    &= \norm{\bbE[u(\cdot,t) - u_{\Dx_L}(\cdot,t)]}_{\mathrm{L}^q(\Omega;\mathrm{L}^p(D))}\\
    &\mathrel{\phantom{=}} + \norm{\sum_{l=0}^L\left(\bbE[u_{\Dx_l}(\cdot,t)-u_{\Dx_{l-1}}(\cdot,t)] - E_{M_l}[u_{\Dx_l}(\cdot,t)-u_{\Dx_{l-1}}(\cdot,t)]\right)}_{\mathrm{L}^q(\Omega;\mathrm{L}^p(D))}\\
    &\leq \norm{\bbE[u(\cdot,t) - u_{\Dx_L}(\cdot,t)]}_{\mathrm{L}^q(\Omega;\mathrm{L}^p(D))}\\
    &\mathrel{\phantom{=}}+ \sum_{l=0}^L\norm{\bbE[u_{\Dx_l}(\cdot,t)-u_{\Dx_{l-1}}(\cdot,t)] - E_{M_l}[u_{\Dx_l}(\cdot,t)-u_{\Dx_{l-1}}(\cdot,t)]}_{\mathrm{L}^q(\Omega;\mathrm{L}^p(D))}\mper
  \end{align*}
  For the first term, note that the function $\bbE[u(\cdot,t)-u_{\Dx_L}(\cdot,t)]$ is deterministic and thus we can use the convergence rate estimate \eqref{eqn: Convergence rate estimate for FVM} to get
  \begin{align*}
    \|\bbE[u(\cdot,t) &- u_{\Dx_L}(\cdot,t)]\|_{\mathrm{L}^q(\Omega;\mathrm{L}^p(D))}\\
    &\leq \norm{u(\cdot,t) - u_{\Dx_L}(\cdot,t)}_{\mathrm{L}^1(\Omega;\mathrm{L}^p(D))}\\
    &\leq \norm{u(\cdot,t) - u_{\Dx_L}(\cdot,t)}_{\mathrm{L}^1(\Omega;\mathrm{L}^1(D))}^{\frac{1}{p}} \norm{u(\cdot,t) - u_{\Dx_L}(\cdot,t)}_{\mathrm{L}^1(\Omega;\mathrm{L}^\infty(D))}^{1-\frac{1}{p}}\\
    &\leq \norm{u_0}_{\mathrm{L}^1(\Omega;\mathrm{L}^\infty(D))}^{1-\frac{1}{p}} \Dx_L^{\frac{1}{2p}}\mper
  \end{align*}
  We now estimate the summands in the second term. Similarly to the proof of \Cref{thm: MCFVM error estimate} we distinguish the two cases $p\geq q$ and $p<q$.
\begin{enumerate}[(1)]
	\item We first consider the case $p\geq q$. Because of the triangle inequality and \eqref{eqn: r-th moment bound of FVM} we have
	\begin{equation*}
		\norm{u_{\Dx_l}(\cdot,t)-u_{\Dx_{l-1}}(\cdot,t)}_{\mathrm{L}^r(\Omega;\mathrm{L}^p(D))} \leq C \norm{u_0}_{\mathrm{L}^r(\Omega;\mathrm{L}^\infty(D))}
	\end{equation*}
	and thus $u_{\Dx_l}(\cdot,t)-u_{\Dx_{l-1}}(\cdot,t)\in \mathrm{L}^r(\Omega;\mathrm{L}^p(D))$. Therefore we can apply \Cref{thm: rand var con est} to get
	\begin{multline*}
		\norm{\bbE[u_{\Dx_l}(\cdot,t)-u_{\Dx_{l-1}}(\cdot,t)] - E_{M_l}[u_{\Dx_l}(\cdot,t)-u_{\Dx_{l-1}}(\cdot,t)]}_{\mathrm{L}^q(\Omega;\mathrm{L}^p(D))}\\
		 \leq C M_l^{\frac{1-q}{q}} \norm{u_{\Dx_l}(\cdot,t)-u_{\Dx_{l-1}}(\cdot,t)}_{\mathrm{L}^q(\Omega;\mathrm{L}^p(D))}\mper
	\end{multline*}
	\item In the case $p<q$, we can apply H\"older's inequality to estimate
	\begin{multline*}
		\norm{\bbE[u_{\Dx_l}(\cdot,t)-u_{\Dx_{l-1}}(\cdot,t)] - E_{M_l}[u_{\Dx_l}(\cdot,t)-u_{\Dx_{l-1}}(\cdot,t)]}_{\mathrm{L}^q(\Omega;\mathrm{L}^p(D))}\\ \leq C \norm{\bbE[u_{\Dx_l}(\cdot,t)-u_{\Dx_{l-1}}(\cdot,t)] - E_{M_l}[u_{\Dx_l}(\cdot,t)-u_{\Dx_{l-1}}(\cdot,t)]}_{\mathrm{L}^q(\Omega;\mathrm{L}^q(D))}\mper
	\end{multline*}
	Following the same steps as in case (2) in the proof of \Cref{thm: MCFVM error estimate} for $u_{\Dx_l}(\cdot,t)-u_{\Dx_{l-1}}(\cdot,t)$ instead of $u(\cdot,t)$ and using \eqref{eqn: r-th moment bound of FVM} instead of \Cref{lem: r-th moment estimate of the entropy solution}, we see that $u_{\Dx_l}(\cdot,t)-u_{\Dx_{l-1}}(\cdot,t)\in \mathrm{L}^q(\Omega;\mathrm{L}^q(D))$. 
	Thus, we can apply \Cref{thm: rand var con est} again and get
	\begin{multline*}
		\norm{\bbE[u_{\Dx_l}(\cdot,t)-u_{\Dx_{l-1}}(\cdot,t)] - E_{M_l}[u_{\Dx_l}(\cdot,t)-u_{\Dx_{l-1}}(\cdot,t)]}_{\mathrm{L}^q(\Omega;\mathrm{L}^q(D))}\\
		 \leq C M_l^{\frac{1-q}{q}} \norm{u_{\Dx_l}(\cdot,t)-u_{\Dx_{l-1}}(\cdot,t)}_{\mathrm{L}^q(\Omega;\mathrm{L}^q(D))}\mper
	\end{multline*}
\end{enumerate}
Combining both cases, we get
  \begin{multline*}
    \norm{\bbE[u_{\Dx_l}(\cdot,t)-u_{\Dx_{l-1}}(\cdot,t)] - E_{M_l}[u_{\Dx_l}(\cdot,t)-u_{\Dx_{l-1}}(\cdot,t)]}_{\mathrm{L}^q(\Omega;\mathrm{L}^p(D))}\\
    \leq C M_l^{\frac{1-q}{q}} \norm{u_{\Dx_l}(\cdot,t) - u_{\Dx_{l-1}}(\cdot,t)}_{\mathrm{L}^q(\Omega;\mathrm{L}^{\widetilde{p}}(D))}
  \end{multline*}
  where $\widetilde{p}=\max(p,q)$.
  Now, we can use the triangle inequality to get
  \begin{multline*}
    \norm{u_{\Dx_l}(\cdot,t)-u_{\Dx_{l-1}}(\cdot,t)}_{\mathrm{L}^q(\Omega;\mathrm{L}^{\widetilde{p}}(D))} \\
    \leq \norm{u_{\Dx_l}(\cdot,t)-u(\cdot,t)}_{\mathrm{L}^q(\Omega;\mathrm{L}^{\widetilde{p}}(D))} + \norm{u(\cdot,t)-u_{\Dx_{l-1}}(\cdot,t)}_{\mathrm{L}^q(\Omega;\mathrm{L}^{\widetilde{p}}(D))}\mper
  \end{multline*}
  For $l>0$, we can use the interpolation inequality between $\mathrm{L}^1$ and $\mathrm{L}^\infty$, the $\mathrm{L}^1$ and $\mathrm{L}^\infty$ bounds of the entropy solution and finite volume approximations (see~\eqref{eqn: Linfty bound of deterministic entropy solution} respectively~\eqref{eqn: Linfty bound of FVM}), and the convergence rate estimate \eqref{eqn: Convergence rate estimate for FVM} to get
  \begin{align*}
  	\|u_{\Dx_l}(\cdot,t) & -u(\cdot,t)\|_{\mathrm{L}^q(\Omega;\mathrm{L}^{\widetilde{p}}(D))} + \norm{u(\cdot,t)-u_{\Dx_{l-1}}(\cdot,t)}_{\mathrm{L}^q(\Omega;\mathrm{L}^{\widetilde{p}}(D))}\\
  	&\leq \norm{u_{\Dx_l}(\cdot,t)-u(\cdot,t)}_{\mathrm{L}^q(\Omega;\mathrm{L}^1(D))}^{\frac{1}{\widetilde{p}}} \norm{u_{\Dx_l}(\cdot,t)-u(\cdot,t)}_{\mathrm{L}^q(\Omega;\mathrm{L}^\infty(D))}^{1-\frac{1}{\widetilde{p}}}\\
  	&\phantom{\mathrel{=}}+ \norm{u(\cdot,t)-u_{\Dx_{l-1}}(\cdot,t)}_{\mathrm{L}^q(\Omega;\mathrm{L}^1(D))}^{\frac{1}{\widetilde{p}}} \norm{u(\cdot,t)-u_{\Dx_{l-1}}(\cdot,t)}_{\mathrm{L}^q(\Omega;\mathrm{L}^\infty(D))}^{1-\frac{1}{\widetilde{p}}} \\
  	&\leq C \norm{u_0}_{\mathrm{L}^q(\Omega;\mathrm{L}^\infty(D))}^{1-\frac{1}{\widetilde{p}}} \left( \Dx_l^{\frac{1}{2\widetilde{p}}} + \Dx_{l-1}^{\frac{1}{2\widetilde{p}}} \right)\\
    &\leq C \norm{u_0}_{\mathrm{L}^q(\Omega;\mathrm{L}^\infty(D))}^{1-\frac{1}{\widetilde{p}}} \Dx_l^{\frac{1}{2\widetilde{p}}}\mper
  \end{align*}
  Similarly, for $l=0$ (note that $u_{\Dx_{-1}}=0$), the convergence rate estimate \eqref{eqn: Convergence rate estimate for FVM} and the bound from \Cref{lem: r-th moment estimate of the entropy solution} give
  \begin{align*}
  	\norm{u_{\Dx_0}(\cdot,t)}_{\mathrm{L}^q(\Omega;\mathrm{L}^{\widetilde{p}}(D))} &\leq \norm{u_{\Dx_0}(\cdot,t) - u(\cdot,t)}_{\mathrm{L}^q(\Omega;\mathrm{L}^{\widetilde{p}}(D))} + \norm{u(\cdot,t)}_{\mathrm{L}^q(\Omega;\mathrm{L}^{\widetilde{p}}(D))} \\
  	&\leq C \left(\norm{u_0}_{\mathrm{L}^q(\Omega;\mathrm{L}^\infty(D))}^{1-\frac{1}{\widetilde{p}}} \Dx_0^{\frac{1}{2\widetilde{p}}} + \norm{u_0}_{\mathrm{L}^q(\Omega;\mathrm{L}^\infty(D))}\right)
  \end{align*}
  Combining all estimates finally gives
  \begin{multline*}
    \norm{\bbE[u(\cdot,t)] - E^L[U(\cdot,t)]}_{\mathrm{L}^q(\Omega;\mathrm{L}^p(\R))}\\
    \leq C \left( \norm{u_0}_{\mathrm{L}^1(\Omega;\mathrm{L}^\infty(D))}^{1-\frac{1}{\widetilde{p}}} \Dx_L^{\frac{1}{2p}} + \norm{u_0}_{\mathrm{L}^q(\Omega;\mathrm{L}^\infty(D))} M_0^{\frac{1-q}{q}} +\norm{u_0}_{\mathrm{L}^q(\Omega;\mathrm{L}^\infty(D))}^{1-\frac{1}{\widetilde{p}}}\sum_{l=0}^L M_l^{\frac{1-q}{q}} \Dx_l^{\frac{1}{2\widetilde{p}}} \right)\mper
  \end{multline*}
\end{proof}

\subsection{Work estimates and sample number optimization}

In order to analyze the efficiency of the MC and MLMCFVM, it is important to estimate the computational work which is needed to compute one approximation of the solution by the deterministic FVM and how it increases with respect to mesh refinement. Here, by computational work, we understand the number of floating point operations performed when executing an algorithm and we assume that this in turn is proportional to the runtime of the algorithm.

In practice, we deal with bounded domains instead of working on the whole real line and thus the number of grid cells scales as $1/\Dx$.
For the deterministic FVM~\eqref{eqn: FVM} the number of floating point operations per time step is proportional to the number of cells in the spatial domain, hence the computational work can be bounded by $C \Dt^{-1}\Dx^{-1}$. Considering the CFL condition \eqref{eqn: CFL condition}, we thus obtain the computational work estimate
\begin{equation*}
	W^{\text{FVM}}(\Dx) \leq C\Dx^{-2}
\end{equation*}
for the deterministic FVM approximation.
However, for the sake of generality, we will in the following only assume that the computational work scales as
\begin{equation}
  W^{\text{FVM}}(\Dx) \leq C\Dx^{-w}
  \label{eqn: Work of FVM}
\end{equation}
for some $w>0$.
As seen before, we have the $\mathrm{L}^p$ convergence rate estimate
\begin{equation*}
 	\norm{u(\cdot,t) - u_\Dx(\cdot,t)}_{\mathrm{L}^p(D)} \leq C\Dx^{\frac{s}{p}}
 \end{equation*}
 (for $s=\frac{1}{2}$)
 which yields the following deterministic convergence rate with respect to work:
 \begin{equation}
 	\norm{u(\cdot,t) - u_\Dx(\cdot,t)}_{\mathrm{L}^p(D)} \leq C \left(W^{\text{FVM}}\right)^{-\frac{s}{wp}}\mper
  \label{eqn: Error rate in terms of work for FVM}
 \end{equation}
 In particular, for $p=1$, $w=2$, and $s=\frac{1}{2}$ we have
  \begin{equation*}
 	\norm{u(\cdot,t) - u_\Dx(\cdot,t)}_{\mathrm{L}^1(D)} \leq C (W^{\text{FVM}})^{-\frac{1}{4}}\mper
 \end{equation*}

\subsubsection{Work estimates for the MCFVM approximation}
Since for the Monte Carlo finite volume method $M$ deterministic finite volume approximations need to be computed, each of which require work as in~\eqref{eqn: Work of FVM}, the computational work for the MCFVM is bounded as
\begin{equation}
	W_M^{\text{MC}} \leq C M \Dx^{-w}\mper
	\label{eqn: Work of MCFVM}
\end{equation}
In order to obtain the order of convergence of the approximation error in terms of computational work, we equilibrate the terms $M^{\frac{1-q}{q}}$ and $\Dx^{\frac{s}{p}}$ in~\eqref{eqn: MCFVM error estimate} by choosing $M= C \Dx^{\frac{sq}{p(1-q)}}$. Inserting this into the work bound~\eqref{eqn: Work of MCFVM} yields
\begin{equation*}
	W_M^{\text{MCFVM}} \leq C \Dx^{\frac{sq-wp(1-q)}{p(1-q)}}
\end{equation*}
such that we obtain from~\eqref{eqn: MCFVM error estimate}
\begin{equation}
	\norm{\bbE[u(\cdot,t)] - E_M[u_\Dx(\cdot,t)]}_{\mathrm{L}^q(\Omega;\mathrm{L}^p(D))} \leq C \Dx^{\frac{s}{p}} \leq C \left( W_M^{\text{MC}} \right)^{-\frac{s}{wp + s\frac{q}{q-1}}}\mper
  \label{eqn: Error rate in terms of work for MCFVM}
\end{equation}
Note that, since $q/(q-1)$ is positive, we have
\begin{equation*}
  \frac{s}{wp + s\frac{q}{q-1}} \leq \frac{s}{wp}
\end{equation*}
and thus the rate~\eqref{eqn: Error rate in terms of work for MCFVM} is worse than the error rate in terms of computational work~\eqref{eqn: Error rate in terms of work for FVM} of the deterministic finite volume method.

In particular, for $p=1$ and $r\geq 2$ (which implies $q=2$), and taking into account that $w=2$ and $s=\frac{1}{2}$, the rate~\eqref{eqn: Error rate in terms of work for MCFVM} reads
\begin{equation*}
	\norm{\bbE[u(\cdot,t)] - E_M[u_\Dx(\cdot,t)]}_{\mathrm{L}^2(\Omega;\mathrm{L}^1(D))} \leq C 
	\left( W_M^{\text{MC}} \right)^{-\frac{1}{6}}\mper
\end{equation*}

\subsubsection{Optimal sample numbers for the MLMCFVM approximation}
In \cite{koley2017multilevel}, Koley et al. showed the following general result for multilevel Monte Carlo finite volume methods which we can apply to our case to determine the number of samples needed at each level $l$ such that, given an error tolerance $\varepsilon>0$, the computational work of the MLMCFVM is minimal.
\begin{lemma}[{\cite[Lem. 4.9]{koley2017multilevel}}]\label{lem: Optimization of MLMC work}
	Assume that the work of a multilevel Monte Carlo finite volume method with~$L$ discretization levels scales asymptotically as
	\begin{equation*}
		W_L^{\text{MLMC}} = C\sum_{l=0}^L M_l \Dx_l^{-w}
		\label{eqn: Work estimate MLMC}
	\end{equation*}
	for some $w>0$ and that the approximation error (raised to the $q$-th power) scales as
	\begin{equation*}
		\text{Err}_L = C\left( \Dx_L^{\frac{sq}{p}} + M_0^{1-q} + \sum_{l=0}^L M_l^{1-q} \Dx_l^{\frac{sq}{\widetilde{p}}} \right) 
	\end{equation*}
	where $\widetilde{p}=\max(p,q)$ (cf.~\eqref{eqn: MLMCFVM error estimate}).
	Then, given an error tolerance $\varepsilon>0$, the optimal sample numbers $M_l$ minimizing the computational work given the error tolerance $\varepsilon$ are given by
	\begin{equation}
		M_0 \simeq \left( \frac{1+\Dx_0^{\frac{s}{\widetilde{p}}}\sum_{l=1}^L 2^{l\left(w\frac{q-1}{q} -\frac{s}{\widetilde{p}}\right)}}{\varepsilon - \Dx_L^{\frac{sq}{p}}} \right)^{\frac{1}{q-1}}
		\label{eqn: Value of M_0}
	\end{equation}
	and
	\begin{equation}
		M_l \simeq M_0 \Dx_0^{\frac{s}{\widetilde{p}}} 2^{-l\left( \frac{s}{\widetilde{p}} + \frac{w}{q} \right)},\qquad \text{for }l>0,
		\label{eqn: Value of M_l}
	\end{equation}
	where $\simeq$ indicates that this is the number of samples up to a constant which is independent of $l$ and $L$. The minimal amount of work then is
	\begin{equation*}
		W_L^{\text{MLMC}} \simeq \Dx_0^{-w} \left( 1+ \Dx_0^{\frac{s}{\widetilde{p}}}\sum_{l=1}^L 2^{l\left(w\frac{q-1}{q} -\frac{s}{\widetilde{p}}\right)} \right) \left( \frac{1+\Dx_0^{\frac{s}{\widetilde{p}}}\sum_{l=1}^L 2^{l\left(w\frac{q-1}{q} -\frac{s}{\widetilde{p}}\right)}}{\varepsilon - \Dx_0^{\frac{sq}{p}}2^{-L\frac{qs}{p}}} \right)^{\frac{1}{q-1}}\mper
	\end{equation*}
\end{lemma}
\Cref{lem: Optimization of MLMC work} can be used to derive a rate for the approximation error of the MLMCFVM in terms of the computational work.
\begin{corollary}
  In addition to the assumptions of \Cref{lem: Optimization of MLMC work}, assume that $w\frac{q-1}{q} -\frac{s}{\widetilde{p}}>0$ and that $L$ and $\Dx_0$ are large enough such that
  \begin{equation*}
    \Dx_L^{\frac{s}{\widetilde{p}}\frac{q}{q-1}-w} > \Dx_0^{-w}
  \end{equation*}
  where $\widetilde{p}=\max(p,q)$ and $w$ is as in~\eqref{eqn: Work estimate MLMC}.
  Then, for each $0\leq t\leq T$ and for $q =\min(2,r)$ the $\mathrm{L}^q(\Omega;\mathrm{L}^p(D))$-approximation error of the MLMCFVM \eqref{eqn: MLMCFVM definition} scales with respect to computational work as
  \begin{equation}
    \norm{\bbE[u(\cdot,t)] - E^L[U(\cdot,t)]}_{\mathrm{L}^q(\Omega;\mathrm{L}^p(D))} \leq C \left(W_L^{\text{MLMC}}\right)^{-\frac{s}{wp + s\frac{\widetilde{p}-p}{\widetilde{p}}\frac{q}{q-1} }}\mper
    \label{eqn: Error in terms of work for MLMC}
  \end{equation}
\end{corollary}
\begin{proof}
Since $\left( w\frac{q-1}{q}-\frac{s}{\widetilde{p}} \right)>0$ the sums in the expression for $W_M^{\text{MLMC}}$ from \Cref{lem: Optimization of MLMC work} are dominated by $2^{L\left( w\frac{q-1}{q}-\frac{s}{\widetilde{p}} \right)}$. Choosing $\varepsilon = 2 \Dx_L^{\frac{sq}{p}}$ and using that $\Dx_L^{\frac{s}{\widetilde{p}}\frac{q}{q-1}-w} > \Dx_0^{-w}$ in the last step, we find
\begin{align*}
  W_L^{\text{MLMC}} &\simeq \Dx_0^{-w} \left( 1+ \Dx_0^{\frac{s}{\widetilde{p}}} 2^{L\left(w\frac{q-1}{q} -\frac{s}{\widetilde{p}}\right)} \right) \left( \frac{1+\Dx_0^{\frac{s}{\widetilde{p}}} 2^{L\left(w\frac{q-1}{q} -\frac{s}{\widetilde{p}}\right)}}{\Dx_L^{\frac{sq}{p}}} \right)^{\frac{1}{q-1}}\\
  &\simeq \Dx_0^{-w} \Dx_L^{-\frac{sq}{p(q-1)}} \left( 1+ \Dx_0^{\frac{s}{\widetilde{p}}} 2^{L\left(w\frac{q-1}{q} -\frac{s}{\widetilde{p}}\right)} \right)^{\frac{q}{q-1}}\\
  &\simeq \Dx_L^{-\frac{sq}{p(q-1)}} \left( \Dx_0^{-w}+ \Dx_L^{\frac{s}{\widetilde{p}}\frac{q}{q-1} -w} \right)\\
  &\simeq \Dx_L^{s\left(\frac{1}{\widetilde{p}}-\frac{1}{p}\right)\frac{q}{q-1} -w}\mper
\end{align*}
Thus, we have
\begin{equation*}
  \norm{\bbE[u(\cdot,t)] - E^L[U(\cdot,t)]}_{\mathrm{L}^2(\Omega;\mathrm{L}^1(D))} = \varepsilon^{\frac{1}{q}} \simeq \Dx_L^{\frac{s}{p}} \simeq \left( W_L^{\text{MLMC}} \right)^{- \frac{s}{wp + s\frac{\widetilde{p}-p}{\widetilde{p}}\frac{q}{q-1}}} \mper
\end{equation*}
\end{proof}
Since $\frac{(\widetilde{p}-p)}{\widetilde{p}}$ and $\frac{q}{(q-1)}$ are nonnegative, we have
\begin{equation*}
	\frac{s}{wp + s\frac{\widetilde{p}-p}{\widetilde{p}}\frac{q}{q-1}} \leq \frac{s}{wp}
\end{equation*}
and thus the error rate in terms of the computational work~\eqref{eqn: Error in terms of work for MLMC} of the MLMCFVM is worse than the error rate~\eqref{eqn: Error rate in terms of work for FVM} for the deterministic scheme.
However, since $\frac{\widetilde{p}-p}{\widetilde{p}} \leq 1-\frac{p}{q}\leq 1$, we have
\begin{equation*}
	\frac{s}{wp + s\frac{\widetilde{p}-p}{\widetilde{p}}\frac{q}{q-1}} \geq \frac{s}{wp + s\frac{q}{q-1}}
\end{equation*}
and thus the error rate~\eqref{eqn: Error in terms of work for MLMC} of the MLMCFVM constitutes an improvement over the (single-level) MCFVM, cf. \eqref{eqn: Error rate in terms of work for MCFVM}.

Note that, in particular, for $p=1$ and $r\geq 2$ (which implies $q=2$ and $\widetilde{p}=2$), and taking into account that $w=2$ and $s=\frac{1}{2}$, the error rate~\eqref{eqn: Error in terms of work for MLMC} reads
\begin{equation*}
	\norm{\bbE[u(\cdot,t)] - E^L[U(\cdot,t)]}_{\mathrm{L}^2(\Omega;\mathrm{L}^1(D))} \leq C 
	\left(W_L^{\text{MLMC}}\right)^{-\frac{1}{5}}\mper
\end{equation*}

\section{Numerical experiments}
\label{sec: numexp}

In this section, we present numerical experiments motivated by two-phase flow in a heterogeneous porous medium\footnote{The code used to produce these experiments can be fount at \url{https://github.com/adrianmruf/MLMC_discontinuous_flux}}. The time evolution of the oil saturation $u\in[0,1]$ can be modeled by~\eqref{eqn: Deterministic Cauchy problem} where the flux is given by
\begin{equation}
	f(k(x),u) = \frac{\lambda_{\text{o}}(u)}{\lambda_{\text{o}}(u) + \lambda_{\text{w}}(u)} (1- k(x)\lambda_{\text{w}}(u)),
	\label{eqn: Buckley--Leverett flux}
\end{equation}
see~\cite[Ex. 8.2]{holden2015front}.
Here, the functions $\lambda_{\text{o}}$ and $\lambda_{\text{w}}$ denote the phase mobilities/relative permeabilities of the oil and the water phase, respectively. Typically, one uses the simple expressions
\begin{equation*}
	\lambda_{\text{o}}(u) = u^2,\qquad \lambda_{\text{w}}(u) = (1-u)^2
\end{equation*}
which we will also do in the subsequent experiments.
The coefficient $k$ in~\eqref{eqn: Buckley--Leverett flux} corresponds to the absolute permeability of the medium. Since the medium is usually layered to some extent throughout the reservoir and even continuously varying geology is typically mapped onto some grid, the coefficient $k$  is often modeled as a piecewise constant function~\cite{gimse1993note}.

Since numerical experiments for conservation laws where the initial datum or the flux is uncertain have been reported in other works (albeit without spatially discontinuous flux), we will here focus on numerical experiments where the discontinuous coefficient $k$ is subject to randomness.
We consider the initial datum
\begin{equation}
	u_0(x) = \begin{cases}
		0.8, & -0.9<x<-0.2,\\
		0.4, & \text{otherwise}
	\end{cases}
	\label{eqn: initial datum}
\end{equation}
on the spatial domain $D=[-1,1]$ with periodic boundary conditions. \Cref{fig: Examples of admissible fluxes} shows two examples of fluxes of the form~\eqref{eqn: Buckley--Leverett flux} and indicates the relevant domain determined by the initial datum~\eqref{eqn: initial datum}.
\begin{figure}
\centering
\begin{tikzpicture}
	\begin{axis}[xtick={0,1}, extra x ticks={0.4,0.8}, extra tick style={grid=major, gray, grid style={gray, dotted}}]
		\addplot[very thick, skyblue1, samples=50, domain=0:1] {(1-2.3*(1-x)^2)*x^2/(x^2+(1-x)^2)};
		\addplot[very thick, plum1, dashed, samples=50, domain=0:1] {(1-0.7*(1-x)^2)*x^2/(x^2+(1-x)^2)};
	\end{axis}
\end{tikzpicture}
\caption{Two possible fluxes of the form~\eqref{eqn: Buckley--Leverett flux} for $k(x)=0.7$ (dashed line) and $k(x)=2.3$ (straight line)}
\label{fig: Examples of admissible fluxes}
\end{figure}
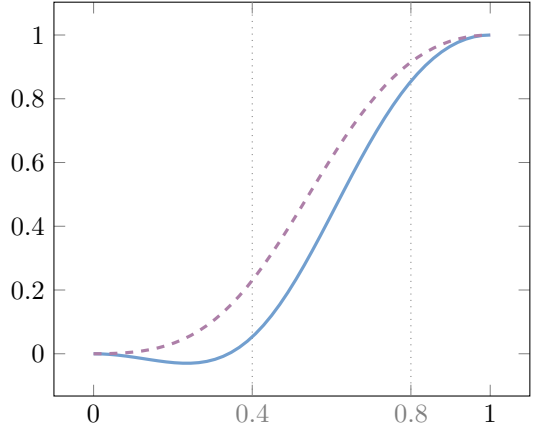
In all experiments we use $\lambda=\frac{\Dt}{\Dx}=0.2$ in the finite volume approximation~\eqref{eqn: FVM}.

When choosing the number of samples for the MLMC estimator we use the formulae~\eqref{eqn: Value of M_0} and~\eqref{eqn: Value of M_l} with $"="$ replacing $"\simeq"$ and rounding to the next biggest integer. Here we use $p=1$, $r=q=2$, $w=2$, $s=\hf$, and $\varepsilon=2\Dx_L^{2s}$ in~\eqref{eqn: Value of M_0} and \eqref{eqn: Value of M_l} \footnote{For example, for $L=7$ and $\Dx_0=2^{-4}$ we use $(M_l)_{l=0}^L = (95646,20107,8454,3555,1495,629,265,112)$ samples.}. 

In order to compute an estimate of the approximation error
\begin{equation*}
	\norm{\bbE[u(\cdot,T)] - E^L[U(\cdot,T)]}_{\mathrm{L}^2(\Omega;\mathrm{L}^1(D))} = \left( \bbE\left[ \norm{\bbE[u(\cdot,T)] - E^L[U(\cdot,T)]}_{\mathrm{L}^1(D)}^2 \right] \right)^\hf
\end{equation*}
we use the root mean square estimator introduced in~\cite{Mishra2012}: We denote by $U_{\text{ref}}(\cdot,T)$ a reference solution and by $(U_i(\cdot,T))_{i=1}^K$ a sequence of independent approximate solutions $E^L[U(\cdot,T)]$ obtained by running the MLMCFVM estimator with $L$ levels $K$ times. Then, we estimate the relative error by
\begin{equation*}
	\mathcal{RMS} = \left( \frac{1}{K} \sum_{i=1}^K \left(\mathcal{RMS}_i\right)^2 \right)^\hf
\end{equation*}
where
\begin{equation*}
	\mathcal{RMS}_i = 100 \times \frac{\norm{U_{\text{ref}}(\cdot,T) - U_i(\cdot,T)}_{\mathrm{L}^1(D)}}{\norm{U_{\text{ref}}(\cdot,T)}_{\mathrm{L}^1(D)}}.
\end{equation*}
Here, as suggested in~\cite{Mishra2012}, we use $K=30$ which was shown to be sufficient for most problems. In order to compute the reference approximation $U_{\text{ref}}(\cdot,T)$ of $\bbE[u(\cdot,T)]$ we take a large number of uniformly-spaced points $(\omega_i)_{i=1}^N$ in $\Omega$ (which in our examples are a closed interval and a rectangle) and compute corresponding finite volume approximations $u_{\Dx^*}(\omega_i;\cdot,T)$ for a very small discretization parameter $\Dx^*$ and then determine $U_{\text{ref}}(\cdot,T)$ by applying the trapezoidal rule to approximate the integral $\int_\Omega u(\omega;\cdot,T)\diff\bbP(\omega)$ using the points $(u_{\Dx^*}(\omega_i;\cdot,T))_{i=1}^N$.

In our experiments we also indicate the approximated standard deviation. To that end, we approximate the variance by
\begin{equation*}
	V_L = \sum_{l=0}^L E_{M_l}\left[ (u_{\Dx_l}(\cdot,T)-u_{\Dx_{l-1}}(\cdot,T) -E_{M_l}[u_{\Dx_l}(\cdot,T) - u_{\Dx_{l-1}}(\cdot,T)])^2 \right].
\end{equation*}

\subsection{Uncertain position of rock layer interface}
For our first numerical experiment we will model the absolute permeability parameter as
\begin{equation*}
	k(x) = \begin{cases}
		1, & x< \xi(\omega),\\
		2, & x> \xi(\omega)
	\end{cases}
\end{equation*}
corresponding to an uncertain position of the interface between two rock types in the reservoir. Here, the random variable $\xi$ is uniformly distributed in $[-0.3,0.3]$.
\begin{figure}[t]
\centering
\subfloat[Two samples of the random entropy solution ($\xi=-0.3$ (straight line), $\xi=0.3$ (dashed line), $\Dx=2^{-9}$).]{
\begin{tikzpicture}
	\begin{axis}[xtick={-1,0,1},ytick={0.4,0.6,0.8},ymin=0.3, ymax=0.9,extra x ticks={-0.3,0.3}, extra tick style={grid=major, gray, grid style={gray, dotted}}]
    	\addplot[skyblue1, very thick] table {Exp1_single_sample_-0.3.txt};
    	\addplot[plum1, dashed, very thick] table {Exp1_single_sample_0.3.txt};
	\end{axis}
\end{tikzpicture}
\label{fig: Experiment 1 single sample}
}
\subfloat[MLMCFVM approximation ($\Dx_0=2^{-4}$, $L=7$).]{
\begin{tikzpicture}
	\begin{axis}[xtick={-1,0,1},ytick={0.4,0.6,0.8},ymin=0.3, ymax=0.9,extra x ticks={-0.3,0.3}, extra tick style={grid=major, gray, grid style={gray, dotted}}]
		\addplot[myorange,name path=mean+std] table {Exp1_mean+std.txt};
    	\addplot[myorange,name path=mean-std] table {Exp1_mean-std.txt};
    	\addplot[fill=myorange] fill between[of=mean+std and mean-std];
    	\addplot[skyblue1,very thick] table {Exp1_mean.txt};
	\end{axis}
\end{tikzpicture}
\label{fig: Experiment 1 MLMCFVM}
}
\caption{Two samples and a MLMCFVM approximation of the (mean of the) random entropy solution for Experiment 1 with $T=0.2$ and $\lambda = 0.2$. The orange area indicates the area between the mean $\pm$ standard deviation. For each sample the discontinuity of $k$ is located in the interval between the dotted lines.}
\end{figure}
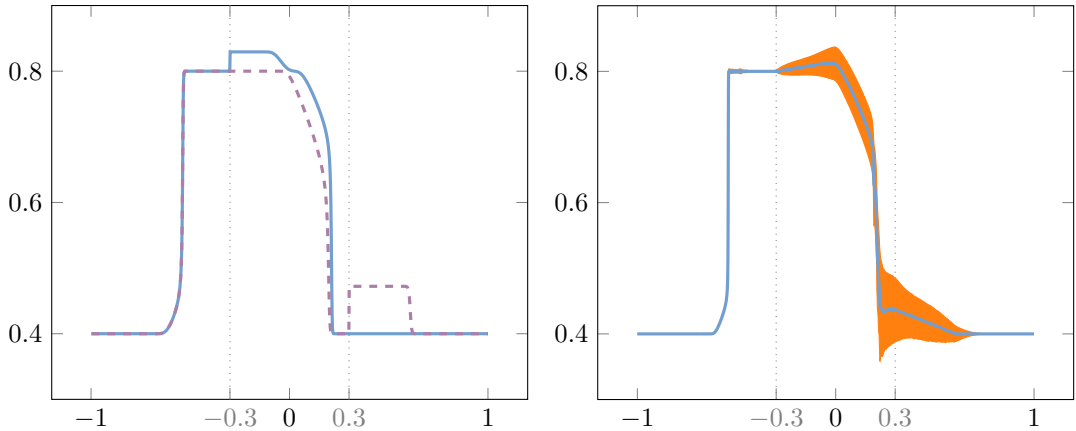
\Cref{fig: Experiment 1 single sample} shows two samples of the approximate random entropy solution (with $\xi = -0.3$ and $\xi=0.3$ respectively) calculated using $2^{10}$ grid points at time $T=0.2$ and \Cref{fig: Experiment 1 MLMCFVM} shows an estimate of the expectation $\bbE[u(\cdot,T)]$ computed by the MLMCFVM with $\Dx_0 = 2^{-4}$ and $L=7$.

\Cref{tab: Experiment 1} and \Cref{fig: convergence rates Experiment 1} show the estimated $\mathcal{RMS}$ error as a function of the number of levels. In particular, \Cref{tab: Experiment 1 rates wrt Dx} shows the observed order of convergence (OOC) with respect to $\Dx_L$ while \Cref{tab: Experiment 1 rates wrt work} shows the observed order of convergence with respect to the computational work calculated based on a best linear fit under the assumptions that $\mathcal{RMS}\sim (\Dx_L)^{r_1}$ and $\mathcal{RMS}\sim (\text{work})^{r_2}$. Here, we use the runtime as a surrogate for the computational work. We observe that in Experiment 1 both rates are better than the rates guaranteed by our convergence analysis.
\begin{table}[t]
\centering
\subfloat[$\mathcal{RMS}$ error versus $\Dx_L$.]{
\begin{tabular}{cccc}
	\toprule
	$L$ & $\Dx_L$ & $\mathcal{RMS}$ & OOC\\
	\midrule
	$1$ & $2^{-5}$  & $4.04$ & \\
	$2$ & $2^{-6}$  & $2.47$ & \\
	$3$ & $2^{-7}$  & $1.44$ & \\
	$4$ & $2^{-8}$  & $0.81$ & \\
	$5$ & $2^{-9}$  & $0.41$ & \\
	$6$ & $2^{-10}$ & $0.17$ & $0.90$\\
	\bottomrule
\end{tabular}
\label{tab: Experiment 1 rates wrt Dx}
}
\hspace{2em}
\subfloat[$\mathcal{RMS}$ error versus work.]{
\begin{tabular}{cccc}
	\toprule
	$L$ & runtime & $\mathcal{RMS}$ & OOC\\
	\midrule
	$1$ & $0.05 $  & $4.04$ & \\
	$2$ & $0.17 $  & $2.47$ & \\
	$3$ & $0.61 $  & $1.44$ & \\
	$4$ & $2.60 $  & $0.81$ & \\
	$5$ & $10.72$  & $0.41$ & \\
	$6$ & $39.64$  & $0.17$ & $-0.46$\\
	\bottomrule
\end{tabular}
\label{tab: Experiment 1 rates wrt work}
}
\caption{$\mathcal{RMS}$ error in Experiment $1$ as a function of the finest grid resolution $\Dx_L$ and as a function of the work (here measured by the runtime in $s$) for various values of $L$ and for $\Dx_0=2^{-4}$.}
\label{tab: Experiment 1}
\end{table}
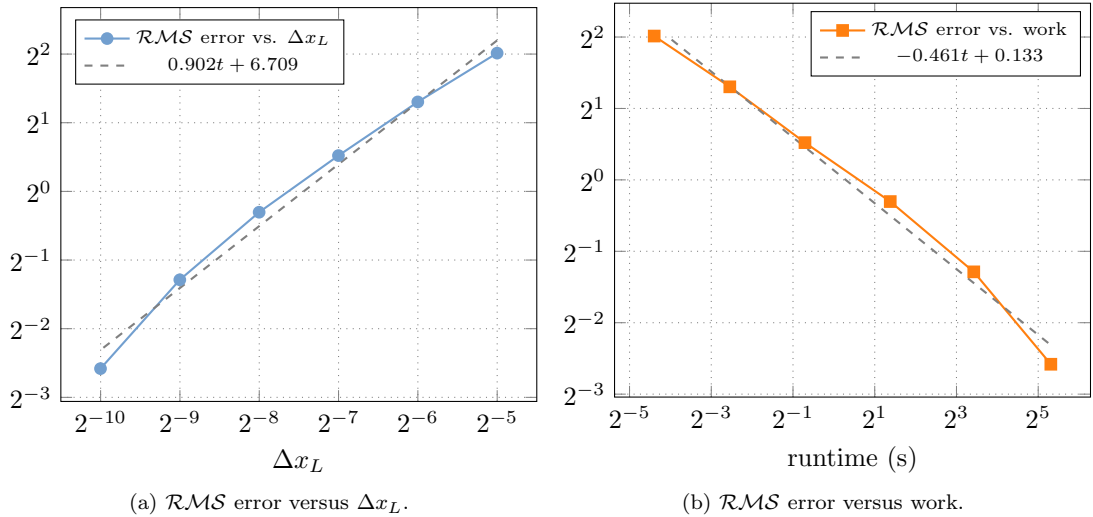
\begin{figure}
\centering
\subfloat[$\mathcal{RMS}$ error versus $\Dx_L$.]{
\begin{tikzpicture}
\begin{loglogaxis}[
    log basis x={2},
    xlabel={$\Dx_L$},
    log basis y={2},
    ylabel={},
    grid=major,
    grid style={gray, dotted},
    legend pos=north west,
]
\addlegendentry{\scriptsize $\mathcal{RMS}$ error vs. $\Dx_L$}
\addplot[mark=*, skyblue1, thick] coordinates {
(0.03125      , 4.03780262292757)
(0.015625     , 2.46696210462274)
(0.0078125    , 1.43538329841369)
(0.00390625   , 0.810277940630155)
(0.001953125  , 0.409640867466471)
(0.0009765625 , 0.167111507119717)
};
\addlegendentry{\scriptsize $0.902t+6.709$}
\addplot[domain=0.03125:0.0009765625, gray, dashed, thick] {2^(0.901979894028114*log2(x) +6.709106754313173)};
\end{loglogaxis}
\end{tikzpicture}
}
\subfloat[$\mathcal{RMS}$ error versus work.]{
\begin{tikzpicture}
\begin{loglogaxis}[
    log basis x={2},
    xlabel={runtime (s)},
    log basis y={2},
    ylabel={},
    grid=major,
    grid style={gray, dotted},
]
\addlegendentry{\scriptsize $\mathcal{RMS}$ error vs. work}
\addplot[mark=square*, myorange, thick] coordinates {
(0.0475067471211111 , 4.03780262292757)
(0.171219623742222  ,  2.46696210462274)
(0.612535331616667  ,  1.43538329841369)
(2.60163620862444   ,   0.810277940630155)
(10.7249629039178   ,   0.409640867466471)
(39.6374996860422   ,   0.167111507119717)
};
\addlegendentry{\scriptsize $-0.461t+0.133$}
\addplot[domain=0.0638:38.8003, gray, dashed, thick] {2^(-0.460944537265283*log2(x) +0.133411918645686)};
\end{loglogaxis}
\end{tikzpicture}
}
\caption{$\mathcal{RMS}$ error in Experiment 1 as a function of the finest grid resolution $\Dx_L$ and as a function of the work (here measured by the runtime in $s$) corresponding to the values in \Cref{tab: Experiment 1}. The dashed lines indicate the observed order of convergence based on a best linear fit.}
\label{fig: convergence rates Experiment 1}
\end{figure}

To compute the reference solution in Experiment 1, we approximated the expectation with respect to the uniform probability distribution on the interval $[-0.3,0.3]$ using the trapezoidal rule with $N=200$ equidistant points and choosing $\Dx^*=2^{-11}$ for the finite volume approximations.

\subsection{Uncertain absolute permeabilities}

For our second numerical experiment we will model the absolute permeability parameter as
\begin{equation*}
	k(x) = \begin{cases}
		1 + \xi_1(\omega), &x<0,\\
		2 + \xi_2(\omega), &x>0
	\end{cases}
\end{equation*}
corresponding to uncertain absolute permeabilities of two rock layers. Here, the random variables $\xi_1$ and $\xi_2$ are both uniformly distributed in $[-0.3,0.3]$.
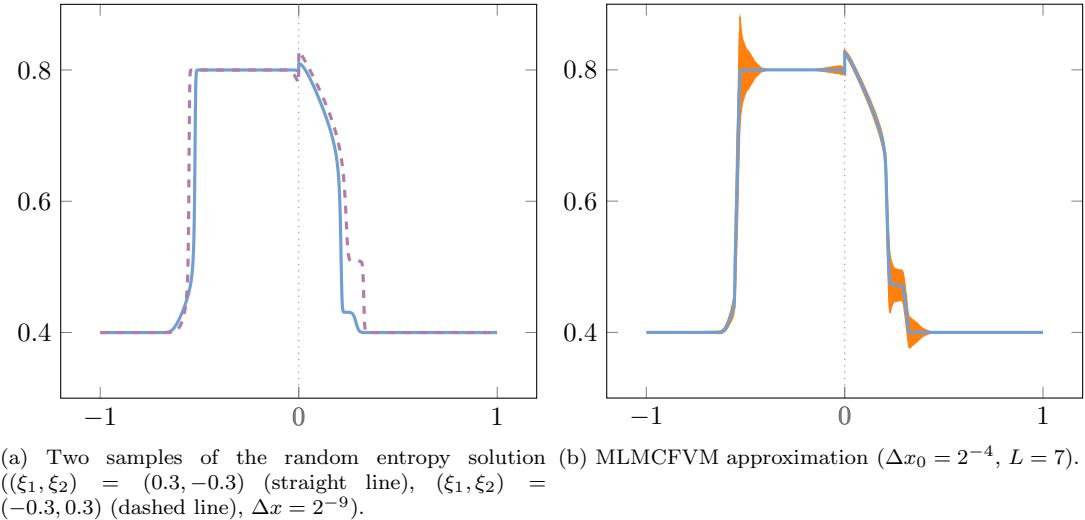
\begin{figure}
\centering
\subfloat[Two samples of the random entropy solution ($(\xi_1,\xi_2)=(0.3,-0.3)$ (straight line), $(\xi_1,\xi_2)=(-0.3,0.3)$ (dashed line), $\Dx=2^{-9}$).]{
\begin{tikzpicture}
	\begin{axis}[xtick={-1,0,1},ytick={0.4,0.6,0.8},ymin=0.3, ymax=0.9,extra x ticks={0}, extra tick style={grid=major, gray, grid style={gray, dotted}}]
    	\addplot[skyblue1, very thick] table {Exp2_single_sample_0.3_-0.3.txt};
    	\addplot[plum1, dashed, very thick] table {Exp2_single_sample_-0.3_0.3.txt};
	\end{axis}
\end{tikzpicture}
\label{fig: Experiment 2 single sample}
}
\subfloat[MLMCFVM approximation ($\Dx_0=2^{-4}$, $L=7$).]{
\begin{tikzpicture}
	\begin{axis}[xtick={-1,0,1},ytick={0.4,0.6,0.8},ymin=0.3, ymax=0.9,extra x ticks={0}, extra tick style={grid=major, gray, grid style={gray, dotted}}]
		\addplot[myorange,name path=mean+std] table {Exp2_mean+std.txt};
    	\addplot[myorange,name path=mean-std] table {Exp2_mean-std.txt};
    	\addplot[fill=myorange] fill between[of=mean+std and mean-std];
    	\addplot[skyblue1, very thick] table {Exp2_mean.txt};
	\end{axis}
\end{tikzpicture}
\label{fig: Experiment 2 MLMCFVM}
}
\caption{Two samples and a MLMCFVM approximation of the (mean of the) random entropy solution for Experiment 2 with $T=0.2$ and $\lambda = 0.2$. The orange area indicates the area between the mean $\pm$ standard deviation and the dotted line marks the (fixed) position of the discontinuity of $k$.}
\end{figure}
\Cref{fig: Experiment 2 single sample} shows two samples of the approximate random entropy solution (with $(\xi_1,\xi_2)=(0.3,-0.3)$ and $(\xi_1,\xi_2)=(-0.3,0.3)$ respectively) calculated using $2^{10}$ grid points at time $T=0.2$ and \Cref{fig: Experiment 2 MLMCFVM} shows an estimate of the expectation $\bbE[u(\cdot,T)]$ computed by the MLMCFVM with $\Dx_0 = 2^{-4}$ and $L=7$.
\begin{table}
\centering
\subfloat[$\mathcal{RMS}$ versus $\Dx_L$.]{
\begin{tabular}{cccc}
	\toprule
	$L$ & $\Dx_L$ & $\mathcal{RMS}$ & OOC\\
	\midrule
	$1$ & $2^{-5}$  & $3.80$ \\
	$2$ & $2^{-6}$  & $2.25$ \\
	$3$ & $2^{-7}$  & $1.34$ \\
	$4$ & $2^{-8}$  & $0.75$ \\
	$5$ & $2^{-9}$  & $0.37$ \\
	$6$ & $2^{-10}$ & $0.15$ & $0.91$ \\
	\bottomrule
\end{tabular}
\label{tab: Experiment 2 rates wrt Dx}
}
\hspace{2em}
\subfloat[$\mathcal{RMS}$ versus work.]{
\begin{tabular}{cccc}
	\toprule
	$L$ & runtime ($s$) & $\mathcal{RMS}$ & OOC\\
	\midrule
	$1$ & $0.05$  & $3.80$ \\
	$2$ & $0.19$  & $2.25$ \\
	$3$ & $0.63$  & $1.34$ \\
	$4$ & $2.70$  & $0.75$ \\
	$5$ & $10.12$  & $0.37$ \\
	$6$ & $38.14$ & $0.15$ & $-0.47$ \\
	\bottomrule
\end{tabular}
\label{tab: Experiment 2 rates wrt work}
}
\caption{$\mathcal{RMS}$ error in Experiment $2$ as a function of the finest grid resolution $\Dx_L$ and as a function of the work (here measured by the runtime in $s$) for various values of $L$ and for $\Dx_0=2^{-4}$.}
\label{tab: Experiment 2}
\end{table}
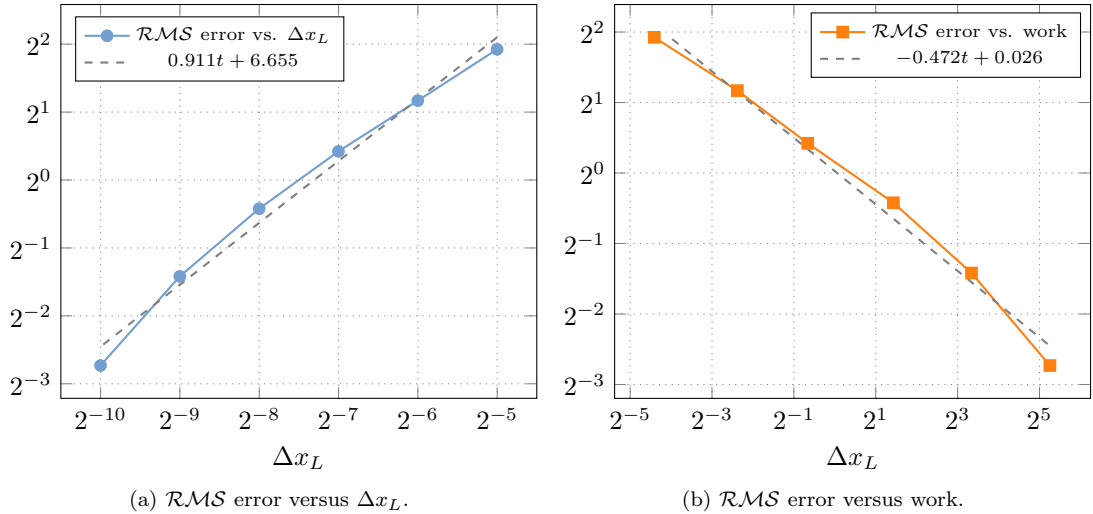
\begin{figure}
\centering
\subfloat[$\mathcal{RMS}$ error versus $\Dx_L$.]{
\begin{tikzpicture}
\begin{loglogaxis}[
    log basis x={2},
    xlabel={$\Dx_L$},
    log basis y={2},
    ylabel={},
    grid=major,
    grid style={gray, dotted},
    legend pos= north west,
]
\addlegendentry{\scriptsize $\mathcal{RMS}$ error vs. $\Dx_L$}
\addplot[mark=*, skyblue1, thick] coordinates {
(0.03125      , 3.79597398255920)
(0.015625     , 2.24789514438061)
(0.0078125    , 1.33955878114555)
(0.00390625   , 0.746532985297767)
(0.001953125  , 0.374108732635936)
(0.0009765625 , 0.150665644359034)
};
\addlegendentry{\scriptsize $0.911t+6.655$}
\addplot[domain=0.03125:0.0009765625, gray, dashed, thick] {2^(0.910852975405070*log2(x) +6.655402745870175)};
\end{loglogaxis}
\end{tikzpicture}
}
\subfloat[$\mathcal{RMS}$ error versus work.]{
\begin{tikzpicture}
\begin{loglogaxis}[
    log basis x={2},
    xlabel={$\Dx_L$},
    log basis y={2},
    ylabel={},
    grid=major,
    grid style={gray, dotted},
]
\addlegendentry{\scriptsize $\mathcal{RMS}$ error vs. work}
\addplot[mark=square*, myorange, thick] coordinates {
(0.0470449738000000 , 3.79597398255920)
(0.192052041266667 , 2.24789514438061)
(0.632521271666667 , 1.33955878114555)
(2.69513145753333 , 0.746532985297767)
(10.1173227042667 , 0.374108732635936)
(38.1413222723000 , 0.150665644359034)
};
\addlegendentry{\scriptsize $-0.472t+0.026$}
\addplot[domain=0.0638:38.8003, gray, dashed, thick] {2^(-0.471782533443006*log2(x) +0.026190490642970)};
\end{loglogaxis}
\end{tikzpicture}
}
\caption{$\mathcal{RMS}$ error in Experiment 2 as a function of the finest grid solution $\Dx_L$ and as a function of the work (here measured by the runtime in $s$) corresponding to the values in \Cref{tab: Experiment 2}. The dotted lines indicate the observed order of convergence based on a best linear fit.}
\label{fig: convergence rates Experiment 2}
\end{figure}

\Cref{tab: Experiment 2} and \Cref{fig: convergence rates Experiment 2} again show the root mean square error estimate and the observed order of convergence with respect to $\Dx_L$ and with respect to the computational work. As before, we observe that the observed convergence rates are better than the theoretical bounds.

In order to compute a reference solution for Experiment 2, we used a tensorized trapezoidal rule with $60\times 60$ points in the stochastic domain $[-0.3,0.3]^2$ and $\Dx^*=2^{-11}$ for the finite volume approximations.

\section{Conclusion}\label{sec: conclusion}
In this paper, we have considered conservation laws with discontinuous flux where the model parameters, i.e., the initial datum, the flux function, and the discontinuous spatial dependency coefficient, are uncertain.
Based on adapted entropy solutions for the deterministic case, we have introduced a notion of random entropy solutions and have proved well-posedness.

To numerically approximate the mean of a random entropy solution, we have proposed Monte Carlo methods coupled with a class of finite volume methods suited for conservation laws with discontinuous flux. Our convergence analysis includes convergence rate estimates for the Monte Carlo and multilevel Monte Carlo finite volume method. Further, we have provided error versus work rates which show that the multilevel Monte Carlo finite volume method is much faster than the (single-level) Monte Carlo finite volume method.

We have presented numerical experiments motivated by two-phase flow in heterogeneous porous media, e.g., oil reservoirs with different rock layers. The numerical experiments verify our theoretical results concerning convergence rates of the multilevel Monte Carlo finite volume method.

As a possible direction of future research, we want to mention that -- from a practical standpoint -- it would be desirable to design multilevel Monte Carlo finite volume methods based on finite volume methods that require no processing of the flux discontinuities. Such numerical methods have been considered in~\cite{TOWERS20205754,ghoshal2020convergence}, however, there are currently no convergence rate results available for these methods.


\bibliographystyle{siam}

\end{document}